\documentclass{amsart}
\usepackage{amsfonts,amsmath,amsthm,amssymb,}
\usepackage[latin1]{inputenc}
\usepackage{enumerate}
\usepackage{enumerate}
\usepackage[breaklinks]{hyperref}
\usepackage{verbatim}
\usepackage[mathscr]{euscript}
\usepackage{enumerate,multicol,array}

\title[Unbounded closed self-adjoint operators in Hilbert spaces]{Model theory of a Hilbert space expanded with an unbounded closed selfadjoint operator}

\author{Camilo Argoty}
\address{Camilo Argoty
\\ Universidad Sergio Arboleda
\\ Departamento de Matem\'aticas
\\ Calle 74 \# 14-14, Bogot\'a, Colombia }
\date{}

\thanks{The author is very thankful to Alexander Berenstein, Andr\'es Villaveces and Pedro Zambrano for his help in reading and correcting this work.}
\makeatletter
\def\newrefformat#1#2{%
  \@namedef{pr@#1}##1{#2}}
\def\prettyref#1{\@prettyref#1:}
\def\@prettyref#1:#2:{%
  \expandafter\ifx\csname pr@#1\endcsname\relax%
    \PackageWarning{prettyref}{Reference format #1\space undefined}%
    \ref{#1:#2}%
  \else%
    \csname pr@#1\endcsname{#1:#2}%
  \fi%
}
\makeatother

\newrefformat{eq}{\textup{(\ref{#1})}}
\newrefformat{cha}{Chapter \ref{#1}}
\newrefformat{sec}{Section \ref{#1}}
\newrefformat{tab}{Table \ref{#1} on page \pageref{#1}}
\newrefformat{fig}{Figure \ref{#1} on page \pageref{#1}}

\makeatletter
\def\indsym#1#2{%
  \setbox0=\hbox{$\m@th#1x$}%
  \kern\wd0%
  \hbox to 0pt{\hss$\m@th#1\mid$\hbox to 0pt{$\m@th#1^{#2}$}\hss}%
  \lower.9\ht0\hbox to 0pt{\hss$\m@th#1\smile$\hss}%
  \kern\wd0} \newcommand{\ind}[1][]{\mathop{\mathpalette\indsym{#1}}}
\def\nindsym#1#2{%
  \setbox0=\hbox{$\m@th#1x$}%
  \kern\wd0%
  \hbox to 0pt{\mathchardef\nn="3236\hss$\m@th#1\nn$\kern1.4\wd0\hss}
  \hbox to 0pt{\hss$\m@th#1\mid$\hbox to 0pt{$\m@th#1^{#2}$}\hss}%
  \lower.9\ht0\hbox to 0pt{\hss$\m@th#1\smile$\hss}%
  \kern\wd0}
\newcommand{\nind}[1][]{\mathop{\mathpalette\nindsym{#1}}}
\makeatother

 \def\bv{\bar v}  
   
 \def\bw{\bar w}
\def\s{\sigma}\def\p{\pi}\def\a{\alpha}
\def\b{\beta}
    
\def\ben{\begin{enumerate}}\def\een{\end{enumerate}}
\def\bdc{\begin{description}}\def\edc{\end{description}}
\def\bitm{\begin{itemize}}\def\eitm{\end{itemize}}
\def\bdf{\begin{defin}}\def\edf{\end{defin}}
\def\bth{\begin{theo}}\def\eth{\end{theo}}
\def\bfc{\begin{fact}}\def\efc{\end{fact}}
\def\bco{\begin{coro}}\def\eco{\end{coro}}
\def\brm{\begin{rem}}\def\erm{\end{rem}}
\def\blm{\begin{lemma}}\def\elm{\end{lemma}}
\def\bnt{\begin{nota}}\def\ent{\end{nota}}
\def\bex{\begin{exe}}\def\eex{\end{exe}}
\def\bpf{\begin{proof}}\def\epf{\end{proof}}
\def\bas{\begin{assum}}\def\eas{\end{assum}}
\def\beq{\begin{equation}}\def\eeq{\end{equation}}
\def\bcl{\begin{cla}}\def\ecl{\end{cla}}

  \def\m{\mu}
  
  \def\K{\mathcal{K}}

 \def\r{\rho}\def\G{\Gamma}
 \def\k{\kappa} \def\s{\sigma} \def\a{\alpha}
\def\b{\beta} \def\L{L}  \def\e{\epsilon} \def\l{\lambda}  \def\L{\Lambda}

\def\H{\mathcal{H}}

\def\O+{\oplus}

\newtheorem{theo}{Theorem}[section]

\newtheorem{coro}[theo]{Corollary}
\newtheorem{lemma}[theo]{Lemma}

\theoremstyle{definition}
\newtheorem{defin}[theo]{Definition}
\newtheorem{fact}[theo]{Fact}
\theoremstyle{remark}
\newtheorem{exe}[theo]{Example}
\newtheorem{rem}[theo]{Remark}
\newtheorem{assum}[theo]{Assumption}
\newtheorem{nota}[theo]{Notation}
\newtheorem*{cla}{Claim}

\newcommand{\N}{\ensuremath{\mathbb{N}}}
\newcommand{\Z}{\ensuremath{\mathbb{Z}}}
\newcommand{\Q}{\ensuremath{\mathbb{Q}}}
\newcommand{\R}{\ensuremath{\mathbb{R}}}
\newcommand{\C}{\ensuremath{\mathbb{C}}}
\newcommand{\LL}{\ensuremath{\mathcal{L}}}
\newcommand{\MM}{\ensuremath{\mathcal{M}}}

\newcommand{\NN}{\ensuremath{\mathcal{N}}}

\begin{document}

\maketitle

\begin{abstract}
We study a closed unbounded self-adoint operator $Q$ acting on a Hilbert space $H$ in the framework of \textit{Metric Abstract Elementary Classes} (MAECS). We build a suitable MAEC for $(H,\G_Q)$, prove it is $\aleph_0$-stable up to perturbations and characterize non-splitting and show it has the same properties as non-forking in superstable first order theories. Also, we characterize equality, orthogonality and domination of (Galois) types in that MAEC.
\end{abstract}

\section{introduction}\label{Introduction}

This paper deals with a complex Hilbert space expanded by a unbounded closed selfadjoint operator $Q$, from the point of view of \textit{Metric Abstract Elementary Classes} (see \cite{HiHir1}). 

Previous works to this paper, can be classified in two kinds. The first one, about model theory of Hilbert spaces expanded with some operators in the frame of continuous logic. The second, about development of a notion of \textit{Abstract Elementary Class} similar to Shelah's (see \cite{She88}), but suitable for analytic structures along with its further analysis. 

For the first kind, previous work go back to Jos\'e Iovino PhD Thesis (see \cite{Io}), where he and C. W. Henson (his advisor) noticed that the structure $(H,0,+,\langle\ |\ \rangle, U)$, where $U$ is a unitary operator, is stable. In \cite{BeBue}, Alexander Berenstein and Steven Buechler gave a geometric characterization of forking in that structure after adding to it the projections determined by the Spectral Decomposition Theorem. Ben Yaacov, Usvyatsov and Zadka (see \cite{BUZ}) worked on the first order continuous logic theory of a Hilbert space with a generic automorphism, and chracterized the generic automorphisms on a Hilbert space as those whose spectrum is the unit circle. Argoty and Berenstein (see \cite{ArBer}) studied the theory of the structure $(H,+,0,\langle|\rangle,U)$ where $U$ is a unitary operator in the case when the spectrum is countable. The author and Ben Yaacov (see \cite{ArBen}), studied the case of a Hilbert space expanded by a normal operator $N$.  Finally  in a recently submitted paper, the author has dealt with non-degenerate representations of an unital (non-commutative) $C^*$-algebra (see \cite{Ar}). 

For the second kind, in $1980'$s S. Shelah defined in \cite{She88} the so called \textit{Abstract Elementary Class} (AEC) as a generalization of the elementary class which is a class of models of a first order theory. As ever, this paper from Shelah generated a big trend in model theory towards the study of this classes. In order to deal with the case of analytic structures, Tapani Hyttinen and $\mathring{\text{A}}$sa Hirvonen defined  \textit{metric abstract elementary classes} in \cite{HiHir1} as a generalization of Shelah's AEC's to classes of metric structures (MAEC's).  After this, in \cite{ViZa1,ViZa2} Villaveces and Zambrano studied notions of independence and superstability for \textit{metric abstract elementary clases} (MAEC's). 

The main results in this paper are the following:
\bitm
	\item We build a MAEC associated with the structure $(H,\G_Q)$ which is denoted by $\K_{(H,\G_Q)}$.
	\item We characterize (Galois) types of vectors in some structure in $\K_{(H,\G_Q)}$, in terms of spectral measures. 
	\item We show that $\K_{(H,\G_Q)}$ is $\aleph_0$-stable up to perturbations.
	\item We characterize non-splitting in $\K_{(H,\G_Q)}$ and we show that it has the same properties as non-forking for superstable first order theories.
\eitm
This paper is divided as follows: In the section \ref{preliminaries}, we give an introduction to Spectral Theory of unbounded closed selfadoint operators. In section \ref{(H-Q)MAEC} In this section we define a \textit{metric abstract elementary class} associated with $(H,\G_Q)$ (denoted by $\K_{(H,\G_Q)}$). In section \ref{definablealgebraicclosures}, we give a characterization of definable and algebraic closures. In section \ref{stability}, we prove superstability of the MAEC $\K_{(H,\G_Q)}$. In section \ref{forking}, we define spectral independence in $\K_{(H,\G_Q)}$ and we show that it is equivalent to non-splitting with the same properties as non-forking for superstable first order theories. Finally in section \ref{orthogonalitydomination}, we characterize domination, orthogonality of types in terms of absolute continuity and mutual singularity between spectral measures.

\section{preliminaries: spectral theory of a closed unbounded self-adjoint operator}\label{preliminaries}
This is a small review of spectral theory of a closed unbounded self-adjoint operator. The main sources for this section are \cite{DuSchw,ReedSi}.

\bdf
Let $H$ be a complex Hilbert space. A \textit{linear operator on} $H$ is a function $Q:D(Q)\to H$ such that $D(Q)$ is a dense vector subspace of $H$ and for all $v$, $w\in D$ and $\a$, $\b\in\C$, $Q(\a v+\b w)=\a Qv+\b Qw$.
\edf

\begin{defin}
Let $Q$ be a linear operator on $H$. The operator $Q$ is called \emph{bounded} if the set $\{\|Qu\| : u\in D(Q), \|u\|=1\}$ is bounded in $\mathbb{C}$. If $Q$ is not bounded, it is called \textit{unbounded}.
\end{defin}

\begin{defin}
If $Q$ is bounded we define the \emph{norm} of $S$ by:
\[\|Q\|=\sup_{u\in D(S),\|u\|=1}\|Su\|\]
For $H$ a Hilbert space, we denote by $B(H)$ the algebra of all bounded linear operators on $H$. 
\end{defin}

\bdf
Let $R$ and $S$ be linear operators on $H$ and let $\a\in\C$. Then the linear operators $R+S$, $\a S$ and $S^{-1}$ are defined as follows:
	\ben
		\item If $D(R)\cap D(S)$ is dense in $H$, $D(R+S):=D(R)\cap D(S)$ and $(R+S)v:=Rv+Sv$ for $v\in D(R+S)$.
		\item $D(RS):=\{v\in H\ |\ v\in D(S)\text{ and }Sv\in D(R)\}$, $(RS)v:=R(Sv)$ if $D(RS)$ is dense and $v\in D(RS)$.
		\item If $\a=0$, then $\a T\equiv 0$ in $H$. If $\a\neq 0$, $D(\a S):=D(S)$ and $(\a S)v:=\a Sv$ if $v\in D(S)$
		\item If $S$ is one-to-one and $SD(S)$ is dense in $H$, $D(S^{-1}):=SD(S)$ and $S^{-1}v:=w$ if $w\in D(S)$ and $Sw=v$
	\een
\edf

\begin{defin}
Let $Q:D(Q)\to H$ be a linear operator on $H$. The operator $Q$ is called \textit{closed} if the set $\{(v,Qv)\ |\ v\in D(Q)\}$ is closed in $H\times H$. The operator $Q$ is called \textit{closable} if the closure of the set $\{(v,Qv)\ |\ v\in D(Q)\}$ is the graph of some operator which is called the \textit{closure} of $Q$ and is denoted by $\bar{Q}$.
\end{defin}

\begin{defin}
Let $Q$ be an operator (either bounded or unbounded), and $\lambda$ a complex number 
\ben
\item $\l$ is called a \textit{eigenvalue} of $Q$ if the operator $Q-\lambda I$ is not one to one. The \textit{point spectrum} of $Q$, denoted by $\s_p(Q)$, is the set of all the eigenvalues of $Q$.
\item $\lambda$ is called a \textit{continuous spectral value} if the operator $Q-\lambda I$ is one to one, the operator $(Q-\lambda I)^{-1}$ is densely defined but is unbounded. The \textit{continuous spectrum} of $Q$ ($\s_c(Q)$) is the set of all the continuous spectral values of $Q$.
\item $\lambda$ is called a \textit{residual spectral value} if $(Q-\lambda I)H$ is not dense in $H$. The \textit{residual spectrum} of $Q$ ($\s_r(Q)$) is the set of all the residual spectral values of $Q$.
\item The \textit{spectrum} of $Q$ ($\s(Q)$) is the union of $\s_p(Q)$, $\s_c(Q)$ and $\s_r(Q)$.
\item The \textit{resolvent set} of $Q$ ($\r(Q)$) is the set $\C\setminus \s(Q)$. If $\l\in \r(Q)$.
\item The \textit{resolvent of} $Q$ at $\l$ is the operator $(Q-\l I)^{-1}$, and is denoted by $R_\l(Q)$.
\een
\end{defin}

\begin{defin}
Given linear operators $Q:D(Q)\to H$ and $Q^\prime:D(Q^\prime)\to H$ on $H$, $Q^\prime$ is said to be an \textit{adjoint operator} of $Q$ if for every $v\in D(Q)$ $w\in D(Q^\prime)$, $\langle Qv|w\rangle=\langle v|Q^*w\rangle$.
\end{defin}

\begin{defin}
Given a linear operator $Q:D(Q)\to H$ and $Q^\prime:D(Q^\prime)\to H$ on $H$, then $Q^\prime$ is said to be the \textit{adjoint operator} of $Q$, denoted $Q^*$, if $Q^\prime$ is maximal adjoint to $Q$ i.e. if $Q^{\prime\prime}$ is and adjoint operator of $Q$ and $Q^\prime\subseteq Q^{\prime\prime}$ then $Q^\prime = Q^{\prime\prime}$.
\end{defin}

\bdf
An operator $Q$ on $H$ is called \textit{symmetric} if $Q\subseteq Q^*$. If $Q=Q^*$, $Q$ is called \textit{selfadjoint}.
\edf

\bth[Lemma XII.2.2 in \cite{DuSchw}]
The spectrum of a self adjoint operator $Q$ is real and for $\l\in\rho(Q)$, the resolvent $R_l(Q)$ is a normal operator with $R_\l(Q)^*=R_{\hat{\l}}(Q)$ and $\|R_\l(Q)\|\leq |Im(\l)|$.
\eth



\bth[Spectral Theorem Multiplication Form, Theorem VIII.4 in \cite{ReedSi}]\label{Theorem&Spectral&Theorem&Multiplication&Form}
Let $Q$ be self adjoint on a Hilbert space $H$ with domain $D(Q)$. Then there are a measure space $(X,\m)$, with $\m$ finite, an unitary operator $U:H\to L^2(X,\m)$, and a real function $f$ on $X$ which is finite a.e. so that,
	\ben
		\item $v\in D(Q)$ if and only if $f(\cdot)(Uv)(\cdot)\in L^2(X,\m)$.
		\item If $g\in U(D(Q))$, then $(UQU^{-1}g)(x)=f(x)g(x)$ for $x\in X$.
	\een 
\eth

\begin{defin}
A self-adjoint operator $Q$ different from the zero operator is called \textit{positive} and we write $Q\ge 0$, if $\langle Qv|v\rangle\ge 0$ for all $v\in\H$.
\end{defin}

\bth[Spectral Theorem-Functional Calculus Form, Theorem VIII.5 in \cite{ReedSi}]\label{Theorem&Spectral&Theorem&Functional&Calculus}
Let $Q$ be a closed unbounded self-adjoint operator on $H$. Then there is a unique map $\p$ from the bounded Borel functions on $\R$ into $B(H)$ such that,
	\ben
		\item $\p$ is an algebraic $^*$-homomorphism.
		\item $\p$ is norm continuous, that is, $\|\p(h)\|_{B(H)}\leq \|h\|_\infty$.
		\item Let $(h_n)_{n\in \N}$ be a sequence of bounded Borel functions with $h_n(x)\to x$ for each $x$ and $|h_n(x)|\leq |x|$ for all $x$ and $n$. Then for any $v\in D(Q)$, $\lim_{n\to\infty}\p(h_n)v=Qv$.
		\item Let $(h_n)_{n\in \N}$ be a sequence of bounded Borel functions. If $h_n\to h$ pointwise and if the sequence $\|h_n\|_\infty$ is bounded, then $\p(h_n)\to\p(h)$ strongly.
		\item If $v\in H$ is such that $Qv=\l v$, then $\p(h)v=h(\l)v$.
		\item If $h\geq 0$, then $\p(h)\geq 0$
	\een	
\eth

\bdf
Let $\Omega$ be a borel measurable subset of $\R$. By $E_\Omega$ we denote the bounded operator $\p(\chi_\Omega)$ according to Theorem \ref{Theorem&Spectral&Theorem&Functional&Calculus}.
\edf

\bfc[Remark after Theorem VIII.5 in \cite{ReedSi}]\label{Fact&Resolution&Of&Identity}
Previously defined projections satisfy the following properties:
	\ben
		\item For every borel measurable $\Omega\subset \R$, $E_\Omega^2=E_\Omega$ and $E_\Omega^*=E_\Omega$.
		\item $E_\emptyset=0$ and $E_{(-\infty,\infty)}=I$
		\item If $\Omega=\cup_{n=1}^\infty\Omega_n$ with $\Omega_n\cap\Omega_m=\emptyset$ if $n\neq m$, then $\sum_{n=1}^\infty E_{\Omega_n}$ converges to $E_\Omega$ in the strong topology.
		\item $E_{\Omega_1}E_{\Omega_2}=E_{\Omega_1\cap\Omega_2}$ (and therefore $E_{\Omega_1}$ commutes with $E_{\Omega_2}$) for all borel measurable $\Omega_1$, $\Omega_2\subseteq\R$.
	\een
\efc

\begin{defin}
The family $\{E_\Omega\ |\ \Omega\subseteq\R\text{ is borel measurable }\}$ described in Fact \ref{Fact&Resolution&Of&Identity} is called the \textit{spectral projection valued measure} (s.p.v.m.) generated by $Q$.
\end{defin}


\begin{fact}[Remark before Theorem VIII.6 in \cite{ReedSi}]\label{esquareisaregularmeasure}
Let $v\in\H$. Then the set function such that for every Borel set $\Omega\subset\R$ assigns the value $\langle E_\Omega v| v\rangle$ is a Borel measure. In the case when $\Omega=(-\infty,\l)$, this measure is denoted $\langle E_\l v| v\rangle$.
\end{fact}

\begin{fact}[Spectral Theorem-Integral Decomposition form, Theorem VIII.6 in \cite{ReedSi}]\label{Theorem&Spectral&Theorem&Integral&Representation&Form}
Let $Q$ be a closed unbounded self-adjoint operator on $H$ and let $h$ be a (possibly unbounded) Borel measurable function on $\R$. Then the (possibly unbounded) operator $h(Q)$ such that for every $v\in H$ 
\[\langle \p(h)v\ | v\rangle :=\int_{-\infty}^\infty h(l)d\langle E_\l v\ |\ v\rangle,\]  
whenever $v\in D(\p(h))$, with
\[D(\p(h)):=\{v\in h\ |\ \int_{-\infty}^\infty|h(l)|^2d\langle E_\l v\ |\ v\rangle<\infty\},\]
is such that $\p(h)$ satisfies properties 1-4 of Theorem \ref{Theorem&Spectral&Theorem&Functional&Calculus} and if $h$ is a bounded borel measurable function on $\R$, then $\p(h)$ is exactly the operator described in Theorem \ref{Theorem&Spectral&Theorem&Functional&Calculus}.
\end{fact}

\begin{defin}
The \textit{essential spectrum} of a closed unbounded self adoint operator $Q$ ($\s_e(Q)$) is the set of complex values $\l$ such that for every bounded operator $S$ on $H$ and every compact operator $K$ on $H$, we have that $(Q-\l I)S\neq I+K$
\end{defin}

\bfc
Let $Q$ be a closed unbounded self adjoint operator on $H$. Then $\s_e(Q)\subseteq \s(Q)\subseteq \R$.
\efc
\bpf
Clear by definition of $\s(Q)$.
\epf

\bth\label{Weyl'sCriterion}
Let $Q$ be a closed unbouned self adjoint operator. Then, for every $\l\in \R$, the following conditions are equivalent:
\begin{enumerate}[i]
    \item  $\l\in\s_e(Q)$
    \item For every $\e>0$, $\dim (E_{(\l-\e,\l+\e)}H)=\infty$
\end{enumerate}
\eth
\begin{proof}
\begin{enumerate}
    \item[(i)$\Rightarrow$(ii)] Asume that there exists $\e>0$ such that $E_{(\l-\e,\l+\e)}H$ finite dimensional. Let 
\[h(x)=\frac{1-\chi_{(\l-\e,\l+\e)}(x)}{x-\l}.\]
Then $h$ is a bounded borel measurable function on $\R$. By Fact \ref{Theorem&Spectral&Theorem&Functional&Calculus} (functional calculus), we have that,
\[f(Q)(Q-\l I)=(Q-\l I)f(Q)=I-\chi_{(\l-\e,\l+\e)}(Q)=I-E_{(l-\e,\l+e)}H\]
Since $E_{(\l-\e,\l+\e)}(Q)$ is finite dimensional, it is compact and $\l\not\in\s_e(Q)$
    \item[(ii)$\Rightarrow$(i)] Suppose that $\l\not\in\s_e(Q)$. Then there are a bounded operator $S$ and a compact operator $K$ such that, 
        \begin{equation}\label{N-lIinvertible2}
S(Q-\l I)=(Q-\l I)S=I+K
        \end{equation}
Suppose that for some $v\in H$, $(Q-\l I)v=0$. Then $(I-K)v=0$ and, therefore, $Kv=-v$. Since $K$ is compact, this implies that $Ker(Q-\l I)$ is finite dimensional
By Hypotesis, for all $\e>0$, $\chi_{(\l-\e,\l+\e)}(Q)$ is infinite dimensional and contains $ker(Q-\l I)$ which is finite dimensional. So, for every $\e>0$ there exists $v_\e\in\chi_{(\l-\e,\l+\e)}(Q)$ such that $\|v_\e\|=1$ and $d(v_\e,ker(Q-\l I))=1$
By Theorem \ref{Theorem&Spectral&Theorem&Integral&Representation&Form}
\begin{multline*}
\|(Q-\l I)v_\e\|^2=\langle (Q-\l I)^*(Q-\l I)\chi_{(\l-\e,\l+\e)}(Q)(v_\e)|v_\e\rangle=\\
=\int_{\l-\e}^{\l+\e}|x-\l|^2d\langle E_x v_\e\ |\ v_\e\rangle\leq \int_{\l-\e}^{\l+\e}|x-\l|^2dx \leq\e^2\int_{\l-\e}^{\l+\e}dx\leq 2\e^3
\end{multline*}
and hence $Qv_\e-\l v_\e\to 0$ when $\e\to 0$. From \eqref{N-lIinvertible2} we get:
\[v_\e+kv_\e=S(Qv_\e-\l v_\e)\to 0 \text{ when }\e\to 0.\]
By compactness of $k$, there exists a sequence $(v_n)\subseteq\{v_\e\ |\ \e>0\}$ such that $kv_n\to v$ when $n\to\infty$ for some $v\in H$. It follows that $v_n\to-v$ and, since $\|v_n\|=1$, we get $\|v\|=1$. Since $Q(v_n)-\l v_n\to 0$ when $n\to \infty$, we get $Qv=\l v$, and hence:
\[\|v_n-v\|\geq d(v_n,ker(Q-\l I))=1,\]
which is a contradiction.
\end{enumerate}
\end{proof}

\bdf
Let $Q$ be a closed unbounded self adjoint operator on $H$. The \textit{discrete spectrum} of $Q$ is the set:
\[\s_d(Q):=\s(Q)\setminus\s_e(Q)\]
\edf

\bdf\label{apequiv}
Let $Q_1$ and $Q_2$ be closed unbounded self adjoint operators defined on Hilbert spaces $H_1$ and $H_2$ respectively. Then $(H_1,\G_{Q_1})$ and $(H_2,\G_{Q_2})$ are said to be \textit{spectrally equivalent} ($Q_1\sim_\s Q_2$) if both of the following conditions hold:
    \begin{enumerate}
		\item $\sigma(Q_1)=\sigma(Q_2)$.
        \item $\sigma_e(Q_1)=\sigma_e(Q_2)$.
        \item $\dim\{x\in H_1\ |\ Q_1x=\lambda x\}=\dim\{x\in H_2\ |\ Q_2x=\lambda x\}$ for $\lambda \in \s(Q_1)\setminus \sigma_e(Q_1)$.
    \end{enumerate}
\edf

\bth[Classical Weyl theorem, Example 3 of Section XIII.4 in \cite{ReedSi}]\label{Classical&Weyl&Theorem}
If $Q$ is a (possibly unbounded) self-adjoint operator and $K$ is a compact operator on $H$. Then $\s_e(Q)=\s_e(Q+K)$.
\eth

\blm[See Lemma II.4.3 in \cite{Da}]\label{Lemma&Equivalence&Dense&Subsets}
Suppose $X$ is a metric space. Let $\{\xi_k\ |\ k\geq 1\}$ and $\{\zeta_k\ |\ k\geq 1\}$ be two countable dense subsets of $X$ such that each isolated point of $X$ is repeated the same number of times in each sequence. Then, given $\e>0$ there is a permutation $\p$ of $\Z_+$ such that dist$(\xi_k,\zeta_{\p(k)})<\e$ for all $k\geq 1$ and 
\[\lim_{k\to\infty}dist(\xi_k,\zeta_{\p(k)})=0\]
\elm

\brm
In \cite{Da}, Lemma II.4.3 states the same thing for $X$ being a compact metric space. However, as we will see, proof works even in the case of a general metric space.
\erm

\bpf[Proof of Lemma \ref{Lemma&Equivalence&Dense&Subsets}]
Without loss of generality, we assume that $X$ has no isolated point of $X$ is repeated a finite number of times. Let $\e>0$ be given. Let $k$ be a positive integer and let us assume by induction that $\p(i)$ and $\p^{-1}(i)$ are already defined for $i=1,\dots,k-1$. 

If $\p(k)$ is not yet defined, we take the least $l\neq\{\p(1),\dots,\p(k-1)\}$ such that 
\[|\xi_k-\zeta_l|<\frac{\e}{2^k}\]
This is possible since $\{\zeta_n\ |\ n\geq 1\ n\neq\{\p(1),\dots,\p(k-1)\}\}$ is still dense in $X$.
Similarly, if $\p^{-1}(k)$ is not yet defined, we take the least $l\neq\{\p^{-1}(1),\dots,\p^{-1}(k-1)\}$ such that 
\[|\xi_l-\zeta_k|<\frac{\e}{2^k}\]
This is possible since $\{\xi_n\ |\ n\geq 1\ n\neq\{\p^{-1}(1),\dots,\p^{-1}(k-1)\}\}$ is still dense in $X$.

Defining $\p$ this way, every $k\in\Z_+$ will be eventually included in the domain and in the range of $\p$ once and only once. So it defines a permutation of $\Z_+$ with the desired properties. 
\epf

\bth[Weyl-Von Neumann-Berg, Corollary 2 in \cite{Be}]\label{Weyl&vonNeumann&Berg}
Let $Q$ be a not necessarilly bounded self adjoint operator on a separable Hilbert space $H$. Then for every $\e>0$ there exists a diagonal operator $D$ and a compact operator $K$ on $H$ such that $\|K\|<\e$ and $Q=D+K$.
\eth

\begin{defin}\label{ApproximatelyUnitarilyEquivalentOperators}
Two unbounded closed self adjoint operators $Q_1$ and $Q_2$ on a separable Hilbert spaces $H_1$ and $H_2$ are said to be \textit{approximately unitarily equivalent} if there exists a sequence of unitary operators $(U_n)_{n<\omega}$ from $H_1$ to $H_2$ such that for every $n\in\Z_+$, $Q_2- U_nQ_1U_n^*$ is compact and for all $\e>0$, there is $n_\e$ such that for every $n\geq n_\e$, $\|Q_2- U_nQ_1U_n^*\|<\e$.
\end{defin}

\bth[See II.4.4 in \cite{Da}]\label{Theorem&Consecuence&Weyl&vonNeumannBerg}
Suppose $Q_1$ and $Q_2$ are unbounded closed self adjoint operators on a separable Hilbert space $H$. Then $Q_1$ and $Q_2$ are approximately unitarily equivalent if and only if $Q_1\sim_\s Q_2$.
\eth

\brm
As in Lemma \ref{Lemma&Equivalence&Dense&Subsets}, we give a proof which is very similar to the one presented in Lemma II.4.3 in \cite{Da}. 
\erm

\bpf[Proof of Theorem \ref{Theorem&Consecuence&Weyl&vonNeumannBerg}]
	\bdc
		\item[$\Rightarrow$] Suppose $Q_1$ and $Q_2$ are approximately unitarily equivalent, and let $(U_n)_{n<\omega}$ be unitary operators from $H$ to $H$ such that for every $n\in\Z_+$, $Q_2- U_nQ_1U_n^*$ is compact and for all $\e>0$, there is $n_\e$ such that for every $n\geq n_\e$, $\|Q_2- U_nQ_1U_n^*\|<\e$. It is clear that $\s(Q_1)=\s(Q_2)$ and if $h$ is a bounded Borel function on $\R$ then 
\[f(Q_2- U_nQ_1U_n^*)=f(Q_2)- U_nf(Q_1)U_n^*\]
is compact and for all $\e>0$, there is $n_\e$ such that for every $n\geq n_\e$, 
\[\|f(Q_2)- U_nf(Q_1)U_n^*\|<\e.\]
In particular, if $f=\chi_{\{x\}}$ is the characteristic function of an isolated point $\s(Q_1)$, then for every $n\in\Z_+$,
\[E_{Q_2}(\{x\})-U_nE_{Q_1}(\{x\})U_n^*=\chi_{\{x\}}(Q_2)-U_n\chi_{\{x\}}(Q_2)U_n^*\]
is compact and for all $\e>0$, there is $n_\e$ such that for every $n\geq n_\e$,
\[\|E_{Q_2}(\{x\})-U_nE_{Q_1}(\{x\})U_n^*\|=\|\chi_{\{x\}}(Q_2)-U_n\chi_{\{x\}}(Q_2)U_n^*\|<\e\]
So, the eigenspace of every isolated point in $\s(Q_1)=\s(Q_2)$ has the same dimension, $\s_e(Q_1)=\s_e(Q_2)$ and $\s_d(Q_1)=\s_d(Q_2)$.
		\item[$\Leftarrow$] Let us suppose at first that $Q_1$ and $Q_2$ are diagonal, and let $Q_1=diag(\xi_k)$ with respect to a basis $v_k$ and $Q_1=diag(\zeta_k)$ with respect to a basis $w_k$. Then $(xi_k)$ and $(\zeta_k)$ are dense in $\s(Q_1)=\s(Q_2)$ and if $\l$ is an isolated point in $\s(Q_1)=\s(Q_2)$, then $\l$ is repeated the same number of times (the dimension of its corresponding eigenspace). Therefore $\s_e(Q_1)=\s_e(Q_2)$.
		
Let $X:=\s_e(Q_1)=\s_e(Q_2)$. Given $\e>0$, by Lemma \ref{Lemma&Equivalence&Dense&Subsets}, there is a permutation $\p$ of $\Z_+$ such that $|\xi_k,\zeta_{\p(k)}|<\e$ for all $k$ and 
\[\lim_{n\to\infty}|\xi_k,\zeta_{\p(k)}|=0\]
Then the unitary operator given by $Uv_k:=w_{\p(k)}$ is such that the operator
\[Q_1-UQ_2U^*=diag(\xi_k-\zeta_{\p(k)})\]
is compact and has norm less than $\e$.

For the more general case, $Q_1$ can be decomposed as $Q=Q_1^d\oplus Q_1^\prime$ such that $\s(Q_1^d)=\s_d(Q_1^d)=\s_d(Q)$, the dimensions of the eigenspaces of elements in $Q_1^d$ are the same as the dimensions of the eigenspaces of values in $\s_d(Q)$ and $\s(Q^\prime)=\s_e(Q)$. The samething happens for $Q_2$ and $Q_1^d$ and $Q_2^d$ coincide. Because of this we can assume that $\s_d(Q)=\emptyset$. Now, given $\e>0$ there are there are diagonal operators $D_1$ and $D_2$ and compact operators $K_1$ and $K_2$ such that $Q_1=D_1+K_1$, $Q_1=D_1+K_1$ and $\|K_1\|$, $\|K_2\|<\e$. By Weyl's theorem (Theorem \ref{Classical&Weyl&Theorem}) $\s(D_1)=\s_e(Q_1)=\s_e(Q_2)=\s(D_2)$. By the diagonal operators case, this implies that $Q_1$ and $Q_2$ are approximately unitarily equivalent.
	\edc
\epf

\bdf
Let $Q$ be a closed unbounded selfadjoint operator on a Hilbert space $H$. For $\l\in\s_d(Q)$, let $n_\l$ be the dimension of the eigenspace corresponding to $\l$. We define the \textit{discrete part} of $H$ in the following way:
\[H_d:=\bigoplus_{\l\in\s_d(Q)}\C^{n_\l}\]
In the same way, we define $Q_d:=Q\upharpoonright H_d$
\edf

\bfc
$H_d\subseteq H$
\efc

\bdf
Let $Q$ be a closed unbounded selfadjoint operator on a Hilbert space $H$. We define the \textit{essential part} of $H$ in the following way:
\[H_e:=H_d^\perp\]
In the same way, we define $Q_e:=Q\upharpoonright H_e$
\edf

\begin{defin}\label{DefinitionHv}
Given $G\subseteq H$ and $v\in H$, we denote by:
\ben 
\item $H_G$, the Hilbert subspace of $H$ generated by the elements $f(Q)v$, where $v\in G$, $f$ is a bounded Borel function on $\R$ and $v\in D(f(Q))$.
\item $Q_G:=Q\upharpoonright H_G$.
\item $H_v$, the space $H_G$ when $G=\{v\}$ for some vector $v\in H$
\item $Q_v:=Q_G$ when $G=\{v\}$.
\item $H_G^\perp$, the orthogonal complement of $H_G$
\item $P_G$, the projection over $H_G$.
\item $P_{G^\perp}$, the projection over $H_G^\perp$.
\een
\end{defin}

\bdf
Given $G\subseteq H$ and $v\in H$, we denote by $(H_G)_d$ and $(H_G)_e$ the projections of $H_G$ on $H_d$ and $H_e$ respectively.
\edf

\bdf
Let $v\in H$, the \textit{spectral measure defined by} $v$ (denoted by $\m_v$) is the finite borel measure that to any borel set $\Omega\subseteq \R$ assigns the (complex) number,
\[\m_v(\Omega):=\langle \chi_\Omega(Q)v\ |\ v\rangle\]
\edf

\blm[Lemma XII.3.1 in \cite{DuSchw}]\label{HvsimL2muv}
For $v\in H$, the space $H_v\simeq L^2(\R,\m_v)$.
\elm

\blm[Lemma XII.3.2 in \cite{DuSchw}]
There is a set $G\subseteq H$ such that $H=\oplus_{v\in G}H_v$.
\elm

\bco\label{Corollary&Decomposition&H}
There is a set $G\subseteq H$ such that $H=H_d\oplus\bigoplus_{v\in G}H_v$
\eco

\section{a metric abstract elementary defined by $(H;Q)$}\label{(H-Q)MAEC}
In this section we define a \textit{metric abstract elementary class} associated with a closed unbounded self-adjoint operator $Q$ defined on a Hilbert space (see Definition \ref{Definition&MAEC}). We will recall several notions related with metric abstract elementary classes that come from \cite{HiHir1}.

\bdf
An $\mathcal{L}$-\textit{metric structure} $\MM$, for a fixed similarity type $\mathcal{L}$, consists of:
\bitm
	\item A closed metric space $(M,d)$
	\item A family $(R^{\MM})_{R\in\mathcal{L}}$ of continuous functions from $M^{n_R}$ into $\R$, where $n_R$ is the arity of $R$.
	\item An indexed family $(F^\MM)_{F\in\mathcal{L}}$ of continuous functions on powers of $M$.
	\item An indexed family $(c^\MM)_{c\in\mathcal{L}}$ of distinguished elements of $M$.
\eitm
We write this structure as
\[\MM=(M,d,(R^{\MM})_{R\in\mathcal{L}},(F^{\MM})_{F\in\mathcal{L}},(c^{\MM})_{c\in\mathcal{L}}).\]
If $\MM$ is a metric structure, $dens(\MM)$ denotes the smallest cardinal of a dense subset of $M$.
\edf

\bdf\label{Definition&Hilbert&Structure}
Let $\mathcal{L}=(0,-,i,+,(I_r)_{r\in\Q},\|\cdot\|,\G_Q)$. A \textit{Hilbert space operator} structure for $\mathcal{L}$ is a metric structure of only one sort: 
\[(H,0,+,i,(I_r)_{r\in\Q},\|\cdot\|,\G_Q)\] 
where
\begin{itemize}
	\item $H$ is a Hilbert space
	\item $Q$ is a closed (unbounded) selfadjoint operator on $H$
	\item $0$ is the zero vector in $H$
	\item $+:H\times  H\to H$ is the usual sum of vectors in $H$
	\item $i: H\to H$ is the function that to any vector $v\in H$ assigns the vector $iv$ where $i^2=-1$
	\item $I_r:H\to H$ is the function that sends every vector $v\in H$ to $rv$, where $r\in\Q$
	\item $\|\cdot\|: H \to \R$ is the norm function
	\item $\G_Q:H\times H\to \R$ is the function that to any $v$, $w\in H$ asigns the number $\G_Q(v,w)$, which is the distance of $(v,w)$ to the graph of $Q$. Since $Q$ is closed, $\G_Q(v,w)=0$ if and only if $(v,w)$ belongs to the graph of $Q$.
\end{itemize}
 Briefly, the structure will be refered to either as $(H,\G_Q)$. 
$(H,\G_Q)$ is a metric structure for the similarity type 
\edf

\blm\label{autounitary}
Let $Q_1$ and $Q_2$ be closed unbounded self adjoint operators defined on Hilbert spaces $H_1$ and $H_2$ respectively. An isomorphism $U:(H_1,\G_{Q_1})\to(H_2,\G_{Q_2})$ is a unitary operator of $U:H_1\to H_2$ such that $UD(Q_1)= D(Q_2)$ and $UQ_1v=Q_2Uv$ for every $v\in D(Q_1)$.
\elm
\begin{proof}
\bdc
	\item[$\Rightarrow$] Suppose $U$ is an isomorphism between $(H_1,\G_{Q_1})$ and $(H_2,\G_{Q_2})$. It is clear that $U$ must be a linear operator. Also, we have that for every $u$, $v\in\H$ we must have that $\langle Uu\,|\,Uv\rangle=\langle u\,|\, v\rangle$ by definition of automorphism. Therefore $U$ must be an isometry and, therefore it must be unitary. 
	
On the other hand, since $U$ is an isomorphism between $(H_1,\G_{Q_1})$ and $(H_2,\G_{Q_2})$, for every $(v,w)\in H\times H$ we have that $\G_{Q_1}(v,w)=\G_{Q_2}(Uv,Uw)$. Therefore, $\G_{Q_1}(v,w)=0$ if and only if $\G_{Q_2}(Uv,Uw)=0$. So, for every $v\in D(Q_1)$, $UQ_1v=Q_2Uv$.
	\item[$\Leftarrow$] Let $U:H_1\to H_2$ be an unitary operator such that $UD(Q_1)= D(Q_2)$ and $UQ_1v=Q_2Uv$ for every $v\in D(Q_1)$. It remains to show that for every $(v,w)\in H\times H$, $\G_{Q_1}(v,w)=\G_{Q_2}(Uv,Uw)$. Let $(v,w)\in H\times H$ be any pair of vectors. There exists a sequence of pairs $(v_n,w_n)_{n\in\N}$ such that for every $n\in\N$, $v_n\in D(Q_1)$, $w_n=Q_1v_n$ and $\G_{Q_1}(v,w)=\lim_{n\to\infty}d[(v,w);(v_n,w_n)]$.
	
By hypothesis, $U$ is an isometry, and maps the graph of $Q_1$ into the graph of $Q_2$; so for all $n\in\N$, $Uv_n\in D(Q_2)$ and $Uw_n=Q_2v_n$. We have that 
\[
\lim_{n\to\infty}d[(Uv,Uw);(Uv_n,Uw_n)]=\lim_{n\to\infty}d[(v,w);(v_n,w_n)]=\G_{Q_1}(v,w).\]
So $\G_{Q_2}(Uv,Uw)\leq \G_{Q_1}(v,w)$. Repeating the argument for $U^{-1}$, we get $\G_{Q_1}(v,w)\leq \G_{Q_2}(Uv,Uw)$.
\edc
\end{proof}

\bfc
Let $(H_1,\G_{Q_1})$, $(H_2,\G_{Q_2})$ and $(H_3,\G_{Q_3})$ such that $Q_1\sim_\s Q_2$. If $(H_2,\G_{Q_2})$ and $(H_3,\G_{Q_3})$ are isomorphic, then $Q_1\sim_\s Q_3$.
\efc

\bdf\label{Definition&MAEC}
A \textit{Metric Abstract Elementary Class} (MAEC), on a fixed similarity type $\mathcal{L}(\mathcal{K})$, is a class $\mathcal{K}$ of $\mathcal{L}(\mathcal{K})$-metric structures provided with a partial order $\prec_{\mathcal{K}}$ such that:
\ben
\item Closure under isomorphism:
	\ben
		\item For every $\MM\in \mathcal{K}$ and every $\mathcal{L}(\mathcal{K})$-structure $\NN$, if $\MM\simeq \NN$ then $N\in \mathcal{K}$.
		\item Let $\NN_1$, $\NN_2\in\mathcal{K}$ and $\MM_1$, $\MM_2\in\mathcal{K}$ be such that there exists $f_l:\NN_l\simeq \MM_l$ (for $l=1,2$) satisfying $f_1\subseteq f_2$. Then $\NN_1\prec_{\mathcal{K}}\NN_2$ implies that $\MM_1\prec_{\mathcal{K}}\MM_2$.
	\een
\item For all $\MM$, $\NN\in\mathcal{K}$ if $\MM\prec_{\mathcal{K}}\NN$ then $\MM\subseteq \NN$.
\item Let $\MM$, $\NN$ and $\MM^*$ be $\mathcal{L}(\mathcal{K})$-structures. If $\MM\subseteq \NN$, $\MM\prec_{\mathcal{K}}\MM^*$ and $\NN\prec_{\mathcal{K}}\MM^*$, then $\MM\prec_{\mathcal{K}}\NN$.
\item Downward L\"owenheim-Skolem: There exists a cardinal $LS(\mathcal{K})\geq \aleph_0+|\mathcal{L}(\mathcal{K})|$ such that for every $\MM\in\mathcal{K}$ and for every $A\subseteq M$ there exists $\NN\in\mathcal{K}$ such that $\NN\prec_{\mathcal{K}}\MM$, $N\supseteq A$ and $dens(N)\leq |A|+LS(\mathcal{K})$.
\item Tarski-Vaught chain:
	\ben
		\item For every cardinal $\m$ and every $\NN\in\mathcal{K}$, if $\{\MM_i\prec_{\mathcal{K}}\NN\ |\ i<\m\}\subseteq \mathcal{K}$ is $\prec_{\mathcal{K}}$-increasing and continuous (i.e. $i<j\Rightarrow \MM_i\prec_{\mathcal{K}}\MM_j$) then $\overline{\bigcup_{i<\m}\MM_i}\in\mathcal{K}$ and $\overline{\bigcup_{i<\m}\MM_i}\prec_{\mathcal{K}}\NN$.
		\item For every $\m$, if $\{\MM_i\ |\ i<\m\}\subseteq \mathcal{K}$ is $\prec_{\mathcal{K}}$-increasing (i.e. $i<j\Rightarrow \MM_i\prec_{\mathcal{K}}\MM_j$) and continuous then $\overline{\bigcup_{i<\m}\MM_i}\in\mathcal{K}$ and for every $j<\m$, $\MM_j\prec_{\mathcal{K}}\overline{\bigcup_{i<\m}\MM_i}$.
	\een
Here, $\overline{\bigcup_{i<\m}\MM_i}$ denotes the completion of $\bigcup_{i<\m}\MM_i$.
\een
\edf

\bdf
Let $(H,\G_Q)$ be a structure as described in Definition \ref{Definition&Hilbert&Structure}. Let $\LL$ the similarity type of $(H,\G_Q)$. We define $\K_{(H,\G_Q)}$ to be the following class:
\[\K_{(H,\G_Q)}:=\{(H^\prime,Q^\prime)\ |\ (H^\prime,Q^\prime)\text{ is an }\LL\text{ Hilbert space operator structure and }Q^\prime\sim_\s Q\}\]
We define the relation $\prec_{\K}$ in $\K_{(H,\G_Q)}$ by:
\[(H_1,\G_{Q_1})\prec_\K(H_2,\G_{Q_2})\text{ if and only if }H_1\subseteq H_2\text{ and }Q_1\subseteq Q_2\]
\edf

\bth\label{Theorem&KHQ&Is&MAEC}
The class $\K_{(H,\G_Q)}$ is a MAEC.
\eth
\bpf
\ben
	\item Closure under isomorphism:
		\ben
			\item Clear by Lemma \ref{autounitary}.
			\item Clear.
		\een
	\item Clear.
	\item Clear.
	\item $LS(\mathcal{K})\leq 2^{2^{\aleph_0}}$. We first prove following claim:
		\bcl
		If $(H^\prime,Q^\prime)\in\K_{(H,\G_Q)}$, there is a $(H^{\prime\prime},Q^{\prime\prime})\prec(H^\prime,Q^\prime)$ such that $(H^{\prime\prime},Q^{\prime\prime})\in\K$ and $|H^{\prime\prime}|\leq 2^{2^{\aleph_0}}$.
		\ecl
		\bpf
		By Corollary \ref{Corollary&Decomposition&H}, there is a set $G^\prime\subseteq H^\prime$ such that $H^\prime=H_d\oplus\bigoplus_{v\in G^\prime}H_v^\prime$. Since there are at most $2^{2^{\aleph_0}}$ many Borel measures, there is a $G^{\prime\prime}\subseteq G$ such that $|G^{\prime\prime}|\leq 2^{2^{\aleph_0}}$ and for every $v\in G^\prime$ there is a $w\in G^{\prime\prime}$ such that $\m_v=\m_w$. Take \[H^{\prime\prime}=H_d\oplus\bigoplus_{v\in G^{\prime\prime}}H_v^\prime\]
 and 
\[Q^{\prime\prime}:=Q^\prime\upharpoonright H^{\prime\prime}\]

Then $(H^{\prime\prime},Q^{\prime\prime})\in\K_{(H,\G_Q)}$, $(H^{\prime\prime},Q^{\prime\prime})\prec(H^\prime,Q^\prime)$ and $|H^{\prime\prime}|\leq 2^{2^{\aleph_0}}$.
		\epf 
		
		Now, let $(H^\prime,Q^\prime)\in\K$ and $A\subseteq H^\prime$. Let $G^\prime$ be as in Corollary \ref{Corollary&Decomposition&H} and let $(H^{\prime\prime},Q^{\prime\prime})$ be as in previous Claim. Since $A\subseteq H_d\oplus\bigoplus_{v\in G^{\prime\prime}}H_v^\prime$, there is a $G_A\subseteq G^{\prime\prime}$ tal que $|G_A|\leq |A|\aleph_0$ such that $A\subseteq H_d\oplus\bigoplus_{v\in G_A}H_v^\prime$. 
		
		Let 
\[\hat{H}:=H_d\oplus\bigoplus_{v\in G_A\cup G^{\prime\prime}}H_v^\prime\]
and
\[Q^{\prime\prime}:=Q^\prime\upharpoonright \hat{H}\]
Then $(\hat{H},\hat{Q})\in\K_{(H,\G_Q)}$, $(\hat{H},\hat{Q})\prec(H^\prime,Q^\prime)$, $A\subseteq \hat{A}$ and $|\hat{H}|\leq |A|+2^{2^{\aleph_0}}$.
	\item Tarski-Vaught chain:
		\ben
			\item Suppose $\kappa$ is a regular cardinal and $(\hat{H},\hat{Q})\in\K_{(H,\G_Q)}$. Let $(H_i,\G_{Q_i})_{i<\kappa}$ a $\prec_\K$ increasing sequence such that $(H_i,\G_{Q_i})\prec_\K(\hat{H},\hat{Q})$ for all $i<\kappa$. Then, for all $i<\kappa$ $(H_{i+1},Q_{i+1})=(H_i,\G_{Q_i})\oplus (H^\prime_i,Q^\prime_i)$, where $H^\prime_i$ is a Hilbert space and $Q^\prime_i)$ is a (possibly unbounded) closed selfadoint operator such that $\s_d(Q^\prime_i)=\emptyset$ and $\s_e(Q^\prime_i)\subseteq\s_e(\hat{Q})$. Then $\overline{\bigcup_{i<\kappa}}(H_i,\G_{Q_i})=H_0\bigoplus_{i<\kappa}(H^\prime_i,Q^\prime_i)$. Since $(H_i,\G_{Q_i})\prec_\K(\hat{H},\hat{Q})$, $\overline{\bigcup_{i<\kappa}}(H_i,\G_{Q_i})\prec_\K(\hat{H},\hat{Q})$.
			\item Clear from previous item.
		\een
\een
\epf

\brm
From now on, the relation $\prec_{\mathcal{K}}$ in $\mathcal{K}_{(H,\G_Q)}$ will be denoted as $\prec$.
\erm

\bdf
Let $(\K,\prec_\K)$ be a MAEC and let $\MM$, $\NN\in\K$ be two structures. An emdedding $f:\MM\to\NN$ such that $f(\MM)\prec_\K\NN$ is called a $\K$-\textit{embedding}. 
\edf

\bdf
A MAEC $\K$ has the \textit{Joint Embedding Property} (JEP) if for any $\MM_1$, $\MM_2\in\K$ there are $\NN\in\K$ and a $\K$-embeddings $f:\MM_1\to\NN$ and $g:\MM_2\to\NN$.
\edf

\bth
$\K_{(H,\G_Q)}$ has the JEP.
\eth
\bpf
Let $(H_1,\G_{Q_1})$, $(H_1,Q_2)\in\K_{(H,\G_Q)}$. Wiothout loss of generality, we can assume that $H_1\cap H_2=\emptyset$. By Corollary \ref{Corollary&Decomposition&H}, there are sets $G_1\subseteq H_1$ and $G_2\subseteq H_2$ such that $H_1=H_d\oplus\bigoplus_{v\in G_1}(H_1)_v$ and $H_2=H_d\oplus\bigoplus_{v\in G_2}(H_2)_v$.

Let
\[\hat{H}=H_d\oplus\bigoplus_{v\in G_1}L^2(\R,\m_v)\oplus\bigoplus_{v\in G_2}L^2(\R,\m_v)\]
By the Spectral Theorem-Multiplication Form and Lemma \ref{HvsimL2muv} for every $v\in G_1\cup G_2$, there is a Borel function $f_v$ and an isomorphism $U_v:H_v\to L^2(\R,\m_v)$ such that $(H_v,Q\upharpoonright H_v)\simeq (L^2(\R,\m_v),M_{f_v})$, where $M_{f_v}$ is the multiplication by $f_v$ in $L^2(\R,\m_v)$.
If
\[\hat{Q}:=(Q_1\upharpoonright H_d)\oplus\bigl(\bigoplus_{v\in G_1}M_{f_v}\bigr)\oplus\bigl(\bigoplus_{v\in G_2}M_{f_v}\bigr)\]
then, $id_{H_d}\oplus\bigoplus_{v\in G_1}U_v$ and $id_{H_d}\oplus\bigoplus_{v\in G_2}U_v$ are respective $\K_{(H,\G_Q)}$-embeddings from $(H_1,\G_{Q_1})$ and $(H_2,\G_{Q_2})$ to $(\hat{H},\hat{Q})$.
\epf

\bdf
A MAEC $\K$ has the \textit{Amalgamation Property} (AP) if for any $\MM$, $\NN_1$, $\NN_2\in\K$ such that $\MM\prec_\K\NN_1$ and $\MM\prec_\K\NN_2$, there are $\MM^\prime\in\K$ and a $\K$-embeddings $f:\NN_1\to\MM^\prime$ and $g:\NN_2\to\MM^\prime$ such that $f(\NN_1)$, $g(\NN_2)\prec_\K\MM^\prime$. and $f\upharpoonright \MM=g\upharpoonright \MM$.
\edf

\bth\label{Theorem&Amalgamation&Property}
$\K_{(H,\G_Q)}$ has the AP.
\eth
\bpf
Let $(H_1,\G_{Q_1})$, $(H_2,\G_{Q_2})$ and $(H_3,\G_{Q_3})\in\K_{(H,\G_Q)}$ be such that $(H_1,\G_{Q_1})\prec (H_2,\G_{Q_2})$ and $(H_1,\G_{Q_1})\prec (H_3,\G_{Q_3})$. By Corollary \ref{Corollary&Decomposition&H}, there are sets $G_1\subseteq H_1$, $G_2\subseteq H_2$ and $G_3\subseteq H_3$ such that:
 
\bitm
	\item $H_1=H_d\oplus\bigoplus_{v\in G_1}(H_1)_v$
	\item $H_2=H_d\oplus\bigoplus_{v\in G_1}(H_1)_v\oplus\bigoplus_{v\in G_2}(H_2)_v$ 			\item $H_3=H_d\oplus\bigoplus_{v\in G_1}(H_1)_v\oplus\bigoplus_{v\in G_3}(H_3)_v$
\eitm

Let 
\[H_4:=H_d\oplus\bigoplus_{v\in G_1}L^2(\R,\m_v)\oplus\bigoplus_{v\in G_2}L^2(\R,\m_v)\oplus\bigoplus_{v\in G_3}L^2(\R,\m_v)\]
and
\[Q_4:=(Q_1\upharpoonright H_d)\oplus\bigl(\bigoplus_{v\in G_1}M_{f_v}\bigr)\oplus\bigl(\bigoplus_{v\in G_2}M_{f_v}\bigr)\oplus\bigl(\bigoplus_{v\in G_3}M_{f_v}\bigr)\]
Then $(H_4,Q_4)\in\K_{(H,\G_Q)}$ and $id_{H_d}\oplus\bigoplus_{v\in G_1}U_v\oplus\bigoplus_{v\in G_2}U_v$, $id_{H_d}\oplus\bigoplus_{v\in G_1}U_v\oplus\bigoplus_{v\in G_3}U_v$ are respective $\K_{(H,\G_Q)}$-embeddings from $(H_2,\G_{Q_2})$ and $(H_3,\G_{Q_3})$ to $(H_4,Q_4)$.
\epf

\brm
For $(H_1,\G_{Q_1})$, $(H_2,\G_{Q_2})$ and $(H_3,\G_{Q_3})$ as in Theorem \ref{Theorem&Amalgamation&Property}, we denote by 
\[(H_2,\G_{Q_2})\bigvee_{(H_1,\G_{Q_1})}(H_3,\G_{Q_3}):=(H_2\vee_{H_2} H_3,Q_2\vee_{Q_1} Q_3)\]
the \textit{amalgamation} of $(H_2,\G_{Q_2})$ and $(H_3,\G_{Q_3})$ over $(H_1,\G_{Q_1})$ as described in Theorem \ref{Theorem&Amalgamation&Property}.
\erm

\bdf\label{Definition&Galois&Type}
For $\MM_1$, $\MM_2\in\K$, $A\subseteq\MM_1\cap\MM_2$ and $(a_i)_{i<\a}\subseteq\MM_1$, $(b_i)_{i<\a}\subseteq\MM_2$, we say that $(a_i)_{i<\a}$ and $(b_i)_{i<\a}$ have the same \textit{Galois type} in $\MM_1$ and $\MM_2$ respectively, $(gatp_{\MM_1}((a_i)_{i<\a}/A)=gatp_{\MM_2}((b_i)_{i<\a}/A))$, if there are $\NN\in\K$ and $\K$-embeddings $f:\MM_1\to\NN$ and $g:\MM_2\to\NN$ such that $(a_i)=g(b_i)$ for every $i<\a$ and $f\upharpoonright A\equiv g\upharpoonright A\equiv Id_A$, where $Id_A$ is the identity on $A$.
\edf

\begin{theo}\label{typeoverempty}\label{typeoverA}
Let $v\in (H_1,\G_{Q_1})$, $w\in (H_2,\G_{Q_2})$ and $G\subseteq H_1\cap H_2$ such that $(H_G,\G_{Q_G})\in\mathcal{K}_{(H,\G_Q)}$, $(H_G,\G_{Q_G})\prec(H_1,\G_{Q_1})$, $(H_G,\G_{Q_G})\prec(H_2,\G_{Q_2})$
$G\subseteq H_1\cap H_2$. Then $gatp_{(H_1,\G_{Q_1})}(v/G)=gatp_{(H_2,\G_{Q_2})}(w/G)$ if and only if 
\[P_Gv=P_Gw\]
and 
\[\m_{P_{G^\perp}v}=\m_{P_{G^\perp}w}.\]
\end{theo}
\bpf
\bdc
	\item[$\Rightarrow$)] Suppose $gatp_{(H_1,\G_{Q_1})}(v/G)=gatp_{(H_2,\G_{Q_2})}(w/G)$ and let $v^\prime:=P_{G^\perp}v$ and $w^\prime:=P_{G^\perp}w$. Then, by Definition \ref{Definition&Galois&Type}, there exists $(H_3,\G_{Q_3})\in\K_{(H,\G_Q)}$ and $\K_{(H,\G_Q)}$-embeddings $U_1:(H_1,\G_{Q_1})\to (H_3,\G_{Q_3})$ and $U_2:(H_2,\G_{Q_2})\to (H_3,\G_{Q_3})$ such that $U_1v=U_2w$ and $U_1\upharpoonright G\equiv U_2\upharpoonright G\equiv Id_G$, where $Id_G$ is the identity on $G$. Since $v=P_Gv+P_{G^\perp}v$, $w=P_Gw+P_{G^\perp}w$ and $U_1\upharpoonright G\equiv U_2\upharpoonright G\equiv Id_G$, we have that $U_1P_Gv=P_Gv$ and $U_2P_Gw=P_Gw$. Since $U_1$ and $U_2$ are embeddings, $\m_{v^\prime}=\m_{U_1v^\prime}=\m_{U_2w^\prime}=\m_{w^\prime}$.
	\item[$\Leftarrow$)] Let $v^\prime:=P_{G^\perp}v$ and $w^\prime:=P_{G^\perp}w$. Suppose $\m_{v^\prime}=\m_{w^\prime}$, then $\m_{v_e^\prime}=\m_{w_e^\prime}$ $L^2(\R,\m_{v_e^\prime})=L^2(\R,\m_{w_e^\prime})$. Let $\m:=\m_{v_e^\prime}=\m_{w_e^\prime}$. Also, let 
\[\hat{H}:=(H_1\vee_{H_G} H_2)\oplus L^2(\R,\m)\]
and let 
\[\hat{Q}:=(Q_1\vee_{Q_G} Q_2)\oplus M_{f_\m}\] be as in the Spectral Theorem-Multiplication form. Let $U_1:(H_1,\G_{Q_1})\to (\hat{H},\hat{Q})$ be the $\K_{(H,\G_Q)}$-embedding acting on $H_{v^\prime}^\perp\vee H_{w^\prime}^\perp$ as in the AP, and acting on $H_{v^\prime}$ as in Lemma \ref{HvsimL2muv}. Define $U_2:(H_2,\G_{Q_2})\to (\hat{H},\hat{Q})$ in the same way. Then we have completed the conditions to show that $gatp_{(H_1,\G_{Q_1})}(v/G)=gatp_{(H_2,\G_{Q_2})}(w/G)$.
\edc
\epf

\bdf
A MAEC $\K$ is said to be \textit{homogeneous} if whenever $\MM$, $\NN\in\K$ and $(a_i)_{i<\a}\subseteq\MM$, $(b_i)_{i<\a}\subseteq\NN$ such that for all $n<\omega$ and $i_0,\dots,i_{n-1}<\a$
\[gatp_{\MM}(a_{i_0},\dots,a_{i_{n-1}}/\emptyset)=gatp_{\NN}(b_{i_0},\dots,b_{i_{n-1}}/\emptyset)\]
then we have that
\[gatp_{\MM}((a_i)_{i<\a}/\emptyset)=gatp_{\NN}((b_i)_{i<\a}/\emptyset),\]
\edf

\bth
$\K_{(H,\G_Q)}$ is an homogeneous MAEC.
\eth
\bpf
Let $(H_1,\G_{Q_1})$, $(H_2,\G_{Q_2})\in\K_{(H,\G_Q)}$ and $(v_i)_{i<\a}\subseteq H_1$, $(w_i)_{i<\a}\subseteq H_2$ be such that for all $n<\omega$ and $i_0,\dots,i_{n-1}<\a$
\[gatp_{(H_1,\G_{Q_1})}(v_{i_0},\dots,v_{i_{n-1}}/\emptyset)=gatp_{(H_2,\G_{Q_2})}(w_{i_0},\dots,w_{i_{n-1}}/\emptyset)\]
Without loss of generality, we can assume that for all $i<\a$ $v_i\in (H_1)_e$ and $w_i\in (H_2)_e$ and for every $i\neq j<\a$, $v_i\perp v_j$ and $w_i\perp w_j$.
For $i<\a$, let $\m_i:=\m_{v_i}=\m_{w_i}$, which is possible by Theorem \ref{typeoverempty}, since for all $i<\a$ $gatp_{(H_1,\G_{Q_1})}(v_i/\emptyset)=gatp_{(H_2,\G_{Q_2})}(w_i/\emptyset)$. Also, let
\[\hat{H}:=(H_1\vee_\emptyset H_2)\oplus \bigoplus_{i<\a}L^2(\R,\m_i)\]
and
\[\hat{Q}:=(Q_1\vee_\emptyset Q_2)\oplus\bigoplus_{i<\a	} M_{f_{\m_i}}\]
be as in the Spectral Theorem-Multiplication form. Let $U_1:(H_1,\G_{Q_1})\to (\hat{H},\hat{Q})$ be the $\K_{(H,\G_Q)}$-embedding acting on $H_{(v_i)_{i<\a}}^\perp\vee H_{(w_i)_{i<\a}}^\perp$ as in the AP, and acting on $H_{(v_i)_{i<\a}}$ as in Lemma \ref{HvsimL2muv}. Define $U_2:(H_2,\G_{Q_2})\to (\hat{H},\hat{Q})$ in the same way. Then we have completed the conditions to show that $gatp_{(H_1,\G_{Q_1})}((v_i)_{i<\a}/\emptyset)=gatp_{(H_2,\G_{Q_2})}((w_i)_{i<\a}/\emptyset)$.
\epf

\bth[Theorem 1.13 in \cite{HiHir1}]
Let $(\K,\prec_\K)$ a MAEC on a similaity type $\L$ satisfying JEP, AP and homgeneity. Let $\kappa>|\L|+LS(\mathcal{K})$, and then there is $\mathfrak{M}\in\K$ such that
\bdc
	\item[$\kappa$-universality] $\mathfrak{M}$ is $\kappa$-universal, that is for all $\MM\in\K$ such that $|\MM|<\kappa$, there is a $\K$ embedding $f:\MM\to\mathfrak{M}$.
	\item[$\kappa$-homogeneity] $\mathfrak{M}$ is $\kappa$-homogeneous, that is if  $(a_i)_{i<\a}$, $(b_i)_{i<\a}\subseteq\mathfrak{M}$ are such that for all $n<\omega$ and $i_0,\dots,i_{n-1}<\a$
\[gatp_{\mathfrak{M}}(a_{i_0},\dots,a_{i_{n-1}}/\emptyset)=gatp_{\mathfrak{M}}(b_{i_0},\dots,b_{i_{n-1}}/\emptyset)\]
then there is an automorphism $f$ of $\mathfrak{M}$ such that $f(a_i)=b_i$ for all $i<\a$.
\edc
\eth

\bdf
If in previous theorem, $\kappa$ is a cardinal greater than the density of any structure in $\K$ that we want to study, the structure $\mathfrak{M}$ is called a \textit{Monster Model}.
\edf

\brm
Let $\kappa$ be as above, and let $\mathbb{M}(\R)$ the set of all regular Borel meaures on $\R$ whoose support is disoint from $\s_p(Q)$. Then the structure $(\tilde{H}_\kappa,\tilde{Q}_\kappa)$ where
\[\tilde{H}=H_d\oplus\bigoplus_\kappa\bigl(\bigoplus_{\m\in\mathbb{M}}L^2(\R,\m)\bigr)\]
and 
\[\tilde{Q}=(Q\upharpoonright H_d)\oplus\bigoplus_\kappa\bigl(\bigoplus_{\m\in\mathbb{M}}M_{f_\m}\bigr)\]
works as a monster model for $\K_{(H,\G_Q)}$. This can be easily proven from the proofs of JEP, AP and homogeneity of $K_{(H,\G_Q)}$.
\erm

\section{definable and algebraic closures}\label{definablealgebraicclosures}
In this section we give a characterization of definable and algebraic closures. Definable closures are described in \ref{definclosure}, while algebraic closures are characterized in \ref{algebraicclosure}.

\bdf
Let $\mathcal{K}$ be a MAEC with JEP and AP. Let $\mathfrak{M}$ be the monster model in $\K$ and let $A\subseteq\mathfrak{M}$ be a small subset. Then,
\ben
	\item The \textit{definable closure} of $A$ is the set
\[dcl(A):=\{m\in\mathfrak{M}\ |\ Fm=m\text{ for all }F\text{ automorphism of }\mathfrak{M}\text{ that fixes }A\text{ pointwise}\}\]
	\item The \textit{algebraic closure} of $A$ is the set
\[acl(A):=\{m\in\mathfrak{M}\ |\ \text{the orbit under all the automorphisms of }\mathfrak{M}\text{ that fix }A\text{ pointwise}\text{ is compact}\}\]
\een
\edf

\begin{theo}\label{definclosure}
Let $G\subseteq \tilde{H}$. Then $dcl(G)=\tilde{H}_G$.
\end{theo}
\begin{proof}
\bdc
	\item[$dcl(G)\subseteq\tilde{H}_G$] Let $v\not\in \tilde{H}_G$. Then $P_{G^\perp}v\neq 0$. Let $(H^\prime,Q^\prime)\in\K_{(H,\G_Q)}$ be a small structure containing $v$. Let $(H^{\prime\prime},Q^{\prime\prime})\in\K_{(H,\G_Q)}$ be a structure containig  $H^\prime\oplus L^2(\R,\m_{P_{G^\perp}v_e})$. Let $w:=P_Gv+(1)_{\m_{P_{G^\perp}v_e}}\in H^{\prime\prime}$. Then $gatp(v/G)=gatp(w/G)$, but $v\neq w$. Therefore $v\not\in dcl(G)$.	\item[$\tilde{H}_G\subseteq dcl(G)$] Let $v\in G$, let $f$ be a bounded Borel function on $\R$, let $U\in Aut(\tilde{H},\tilde{Q}/G)$ and let $(H^\prime,Q^\prime)$ a small structure containg $G$. Then, by Lemma \ref{autounitary}, $Uf(Q^\prime)v=f(Q^\prime)Uv=f(Q^\prime)v$, and $v\in dcl(G)$.
\edc
\end{proof}

\begin{lemma}
Let $v\in\tilde{H}$. If $v$ is an eigenvector corresponding to some $\l\in\s_d(N)$ then $v$ is algebraic over $\emptyset$.
\end{lemma}
\begin{proof}
$\l\in\s_d(N)$ if and only if $\l$ is isolated in $\s(N)$ with finite dimensional eigenspace $\tilde{H}_\l$. So any automorphism can only send $\tilde{H}_\l$ onto $\tilde{H}_\l$ and the orbit of $v$ under such automorphism can only be compact.
\end{proof}

\begin{lemma}\label{algebraicdiscretespectrum}
Let $v\in\tilde{H}$ be such that $v=\sum v_k$ where each $v_k$ is an eigenvector for some $\l_k\in\s_d(N)$. Then $v$ is algebraic over $\emptyset$.
\end{lemma}
\begin{proof}
Given that $\|v_k\|\to 0$ when $k\to\infty$, the orbit of $v$ under all the automorphisms is a Hilbert cube which is compact.
\end{proof}

\bth\label{algebraicclosure}
$acl(\emptyset)=H_d$
\eth
\bpf
$acl(\emptyset)\subseteq H_d$ is a consecuence of Lemma \ref{algebraicdiscretespectrum}. For the converse, suppose $v\in\tilde{H}$ such that $v_e\neq 0$. Let $\kappa$ be an uncountable small cardinal and let $G:=\bigoplus_\kappa L^2(\R,\m_{v_e})$. Any structure in $\K_{(H,\G_Q)}$ containing $G$ will have $\kappa$ different realizations of $gatp(v/\emptyset)$. Therefore $v\not\in acl(\emptyset)$.
\epf

\begin{theo}\label{algebraicclosure}
Let $G\subseteq\tilde{H}$. Then $acl(G)$ is closed Hilbert subspace generated by the union of $dcl(G)$ with $acl(\emptyset)$. 
\end{theo}
\begin{proof}
Let $E$ be the space $acl(\emptyset)+dcl(G)$. We have that $acl(\emptyset)\subseteq acl(G)$ and $dcl(G)\subseteq acl(G)$ so $E\subseteq acl(G)$. If $v\not\in E$, then $P_E^\perp v\neq 0$. Let $\kappa$ be an uncountable small cardinal and let $G:=\bigoplus_\kappa L^2(\R,\m_{(P_E^\perp v)_e})$. Any structure in $\K_{(H,\G_Q)}$ containing $G$ will have $\kappa$ different realizations of $gatp(v/G)$. Therefore,  $v\not\in acl(A)$.
\end{proof}

\section{perturbations}\label{perturbations}
In this section, perturbations of a structure $(H,\G_Q)\in\K_{(H,\G_Q)}$ are defined. Main results here are Theorem \ref{PerturbationProperty} and Theorem \ref{MAECwithPerturbations} that state that $\K_{(H,\G_Q)}$ has the perturbation property and is a MAEC with perturbations respectively.

\bdf
Let $\K$ be a MAEC that satisfies the JEP, AP and homogeneity. Let $\mathfrak{M}$ be its monster model. Then $\K$ is said to have the \textit{perturbation property} if whenever $A\subseteq\mathfrak{M}$ and $(b_i)_{i<\omega}$ is a convergent sequence with limit $b=\lim_{n\to\infty}b_i$ such that $gatp(b_i/A)=gatp(b_j/A)$ for all $i$, $j<\omega$, then $gatp(b/A)=gatp(b_i/A)$ for all $i<\omega$.
\edf

\bth\label{PerturbationProperty}
$\K_{(H,\G_Q)}$ has the perturbation property.
\eth
\bpf
Let $G\subseteq\tilde{H}$ be small and $(v_i)_{i<\omega}\subseteq\tilde{H}$ a sequence such that $\lim_{i\to\infty}v_i=v$ and $gatp(v_i/G)=gatp(v_j/G)$ for all $i$, $j<\omega$. Then by Theorem \ref{typeoverA}, $P_Gv_i=P_Gv_j$ and $gatp(P_{G^\perp}v_i/\emptyset)=gatp(P_{G^\perp}v_j/\emptyset)$ for all $i$, $j<\omega$. If $\lim_{i\to\infty}v_i=v$, it is clear that $P_Gv_i=P_Gv$ for all $i<\omega$. So it is enough to prove the theorem for the case $G=\emptyset$.

Suppose $\lim_{i\to\infty}v_i=v$ and $gatp(v_i/\emptyset)=gatp(v_j/\emptyset)$ for all $i$, $j<\omega$. By Theorem \ref{typeoverempty}, this means that $\m_i=\m_j$ for all $i$, $j<\omega$. Let $\m:=\m_i$ and $E\subseteq\R$ be a Borel set. Then $\langle \chi_E(Q)v\ |\ v\rangle=\langle \chi_E(Q)(\lim_{i\to\infty}v_i\ |\ \lim_{i\to\infty}v_i\rangle=\lim_{i\to\infty}\langle \chi_E(Q)v_i\ |\ v_i\rangle=\lim_{i\to\infty}\m_i(E)=\lim_{i\to\infty}\m(E)=\m(E)$. Again by Theorem \ref{typeoverempty}, $gatp(v_i/\emptyset)=gatp(v/\emptyset)$ for all $i<\omega$.
\epf

\bdf
Let $(\K,\prec_\K)$ be a MAEC. A class $(\mathbb{F}_e)_{e\geq 0}$ collections of bijective mappings between members of $\K$ is said to be a \textit{system of perturbations} for $(\K,\prec_\K)$ if
	\ben
		\item The $\mathbb{F}_\e$ are collections of bijective mappings between members of $\K$ such that
		\item $\mathbb{F}_\delta\subseteq \mathbb{F}_\e$ if $\delta<\e$, $\mathbb{F}_0=\bigcup_{e>0}\mathbb{F}_\e$ and $\mathbb{F}_0$ is exactly the collection of real isomorphisms of structures in $\K$.
		\item If $f:\MM\to\NN$ is in $\mathbb{F}_\e$, then $f$ is a $e^\e$-bi lipschitz mapping with respect to the metric i.e. $e^{-\e}d(x,y)\leq d(f(x),f(y))\leq e^\e d(x,y)$ for all $x$, $y\in M$.
		\item If $f\in \mathbb{F}_\e$ then $f^{-1}\in \mathbb{F}_\e$.
		\item If $f\in \mathbb{F}_\e$, $g\in\mathbb{F}_\delta$, and $dom(g)=rng(f)$ then $g\circ f\in\mathbb{F}_{\e+\delta}$.
		\item If $(f_i)_{i<\a}$ is an increasing chain of $\e$-isomorphisms, i.e. $f_i\in\mathbb{F}_\e$, $f_i\MM_i\to\NN_i$, $\MM_i\prec_\K\MM_{i+1}$, $\NN_i\prec_\K\NN_{i+1}$ and $f_i\subseteq f_{i+1}$ for every $i<\a$, then there is an $\e$-isomorphism $f:\overline{\bigcup_{i<\a}}\MM_i\to\overline{\bigcup_{i<\a}}\NN_i$ such that $f\upharpoonright \MM_i=f_i$ for all $i<\omega$.
	\een			
If $(\mathbb{F}_e)_{e\geq 0}$ is a system of perturbations for $(\K,\prec_\K)$, then $(\K,\prec_\K,(\mathbb{F}_e)_{e\geq 0})$ is called a \textit{MAEC with perturbations}.
\edf

\bdf
Let $\e>0$. An $\e$-\textit{perturbation} in $\K_{(H,\G_Q)}$ is an unitary operator $U:H_1\to H_2$ such that there are closed unbounded selfadoint operators $Q_1$ and $Q_2$ defined on $H_1$ and $H_2$ respectively, such that 
	\ben
		\item $(H_1,\G_{Q_1})$, $(H_2,\G_{Q_2})\in\K_{(H,\G_Q)}$
		\item $UD(Q_1)=D(Q_1)$
		\item The operator $Q_1-U^{-1}Q_2U$ can be extended to a bounded operator on $H_1$ with norm less than $\e$
		\item The operator $Q_2-UQ_1U^{-1}$ can be extended to a bounded operator on $H_2$ with norm less than $\e$
	\een
The class of all $\e$-perturbations in $\K_{(H,\G_Q)}$ is denoted by $\mathbb{F}^{(H,\G_Q)}_e$
\edf

\bth\label{MAECwithPerturbations}
$(\K_{(H,\G_Q)},\prec_{\K_{(H,\G_Q)}},(\mathbb{F}^{(H,\G_Q)}_\e)_{\e\geq 0})$ is a MAEC with perturbations.
\eth
\bpf
Items (1), (2), (3) and (4) are clear. (5) Comes from triangle inequality. Finally, For (6), recall from the Tarsky chain condition in Theorem \ref{Theorem&KHQ&Is&MAEC} that $\overline{\bigcup_{i<\kappa}}(H_i,\G_{Q_i})=H_0\bigoplus_{i<\kappa}(H^\prime_i,Q^\prime_i)$. This with the fact that a direct sum of $\kappa$ bounded operators with norm less than $\e$ is still a bounded operator with norm less than $\e$.		
\epf

\section{stability}\label{stability}
Here, we prove superstability of the MAEC $\K_{(H,\G_Q)}$ by counting types over sets and show that it is $\aleph_0$-stable up to perturbations. This are the statements of Theorem \ref{superstability} and Theorem \ref{StableUpToPerturbations} respectively.

\begin{theo}\label{HvinHw}
Let $v,w\in\tilde{H}$. Then $\tilde{H}_v$ is isometrically isomorphic to a Hilbert subspace of $\tilde{H}_w$ if and only if $\m_v <<\m_w$.
\end{theo}
\begin{proof}
By Radon Nikodim Theorem, if $\m_u <<\m_v$ then $\tilde{H}_v$ is isometrically equivalent to a Hilbert subspace of $\tilde{H}_w$. For the converse, if $\tilde{H}_v$ is isometrically equivalent to a Hilbert subspace of $\tilde{H}_w$, then $v$ can be represented in $L^2(\R,\m_w)$ by some function, and therefore, $\m_u <<\m_v$.
\end{proof}

\brm
Recall that if $G\subseteq\tilde{H}$ is small, $S(G)$ denotes the set of (1) Galois types over $G$.
\erm

\begin{theo}\label{distanceoftypesoverempty1}
Let $p,q\in S(\emptyset)$ and let $v,w\in\tilde{\H}$ such that $v\models p$ and $w\models q$, and $\m_v <<\m_w$. Then, $d(p,q)=\|\m_w-\m_v\|$
\end{theo}
\begin{proof}
If $\m_u <<\m_v$, by Theorem \ref{HvinHw}, there exist $v^\prime\models tp(v/\emptyset)$ such that $\tilde{H}_{v^\prime}\leq\tilde{H}_w$ and there exists $f\in L^1(\s(N),\m_w)$ such that $d\m_v=fd\m_w$. Then $d|\m_w-\m_v|=|1-f|d\m_w$ and therefore $d(p,q)=\|\m_w-\m_v\|$.
\end{proof}

\begin{theo}\label{distanceoftypesoverempty2}
Let $p,q\in S(\emptyset)$ and let $v,w\in\tilde{\tilde{H}}$ be such that $v\models p$ and $w\models q$, and $\m_v \perp\m_w$. Then, $d(p,q)=\sqrt{\|\m_v\|^2+\|\m_w\|^2}$
\end{theo}
\begin{proof}
If $\m_v \perp\m_w$, by Theorem \ref{HvinHw}, neither $\tilde{H}_v$ is not isometrically isomorphic to a Hilbert subspace of $\tilde{H}_w$ nor $\tilde{H}_w$ is isometrically isomorphic to a Hilbert subspace of $\tilde{H}_v$. Then we can assume $\tilde{H}_v\perp\tilde{H}_w$ and therefore, $d(p,q)=\|v-w\|=\sqrt{\|v\|^2+\|w\|^2}=\sqrt{\|\m_v\|^2+\|\m_w\|^2}$.
\end{proof}

\begin{theo}\label{distanceoftypesoverempty3}
Let $p,q\in S(\emptyset)$ and let $v,w\in\tilde{\H}$ be such that $v\models p$ and $w\models q$, and $\m_w=\m_w^\parallel+\m_w^\perp$ according to Lebesgue decomposition theorem. Then, $d(p,q)=\sqrt{\|\m_v-\m_w^\parallel\|^2+\|\m_w^\perp\|^2}$
\end{theo}
\begin{proof}
By Theorem \ref{distanceoftypesoverempty1} and Theorem \ref{distanceoftypesoverempty2}.
\end{proof}

\begin{theo}\label{distanceoftypes}
Let $G\subseteq\tilde{H}$ be small, let $p,q\in S(G)$ and let $v,w\in\tilde{H}$ be such that $u\models p$ and $v\models q$. Then, 
\[d(p,q)=\sqrt{[P_(v)-P_G(w)]^2+d^2(gatp(P_G^\perp v/\emptyset),gatp(P_G^\perp w/\emptyset)}\]
\end{theo}
\begin{proof}
By Theorems \ref{typeoverA}
\end{proof}

\begin{coro}\label{numberoftypes}
Let $G\subseteq \tilde{H}$ then $dens[S_1(F)]\leq |F|\times 2^{\aleph_0}$
\end{coro}
\begin{proof}
Clear from Theorem \ref{typeoverempty}, Theorem \ref{typeoverA} and Theorem \ref{distanceoftypes}.
\end{proof}

\begin{theo}\label{superstability}
$\K_{(H,\G_Q)}$ is $\k$-stable for $\k\geq |\s|$.
\end{theo}
\begin{proof}
Clear from Corollary \ref{numberoftypes}.
\end{proof}

\bdf
A MAEC $\K$ is said to be $\aleph_0$-\textit{stable up to perturbations} if for every pair of separable structure $\MM\prec_\K\NN$, every type $p\in S(\MM)$ and every $\e>0$, there is a separable structure $\NN^\prime$ and an $\e$-perturbation $f:\NN\to\NN^\prime$ such that $p$ is realized in $\NN^\prime$ and $f$ is a ($0$)isomorphism over $\MM$.
\edf

\bth\label{StableUpToPerturbations}
$\K_{(H,\G_Q)}$ is $\aleph_0$-stable up to perturbations.
\eth
\bpf
Let $(H_0,Q_0)\prec (H_1,\G_{Q_1})\in\K_{(H,\G_Q)}$, and let $p\in S(H_0)$. Let $v\in\tilde{H}$ be a realization of $p$ in the monster model. Since $(H_0,Q_0)\oplus (L^2(\R,\m_{v_e}),M_{f_{v_e}}$ and $(H_1,\G_{Q_1})$ are separable and spectrally equivalent, by Theorem \ref{Theorem&Consecuence&Weyl&vonNeumannBerg}, they are approimately uniformly equivalent and therefore there is an $\e$-perturbation relating $(H_1,\G_{Q_1})$ and $(H_0,Q_0)\oplus (L^2(\R,\m_{v_e}),M_{f_{v_e}}$.
\epf

\section{spectral independence}\label{forking}
In this section we define an independence relation in $\K_{(H,\G_Q)}$, called \textit{spectral independence}. Theorem \ref{explicitnonforkingrelation} states that this relation has the same properties as non-forking for superstable firstorder theories, while Theorem \ref{TheoremSplittingImpliesSpectralIndependence} and Theorem \ref{Theorem&Spectral&Independence&Implies&Splitting} state that this relation characterize non-splitting.

\begin{defin}\label{AindependentB}
Let $v\in\tilde{H}$ and let $F$, $G\subseteq\tilde{H}$. We say that $v$ is \textit{spectrally independent} from $G$ over $F$ if $P_{acl(F)}v=P_{acl(F\cup G)}v$ and denote it $v\ind^*_F G$.
\end{defin}

\begin{rem}\label{independenceoverempty}
Let $v$, $w\in\tilde{H}$. Then $v$ is independent from $w$ over $\emptyset$ if and only if $\tilde{H}_{v_e}\perp\tilde{H}_{w_e}$ and denote it $v\ind^*_\emptyset w$. 
\end{rem}

\begin{rem}\label{independenceoverA}
Let $v$, $w\in\tilde{H}$. Let $G\subseteq\tilde{H}$ be small. Then $v$ is independent from $w$ over $G$ if and only if $\tilde{H}_{P^\perp_{acl(G)}(v)}\perp \tilde{H}_{P^\perp_{acl(G)}(w)}$ and denote it $v\ind^*_G w$.
\end{rem}

\brm\label{RemarkItIsEnoughIndependenceForOneVector}
Let $\bv\in H^n$ and $E$, $F\subseteq H$. Then $\bv\ind^*_EF$ if and only if for every $j=1,\dots,n$ $v_j\ind^*_EF$ that is, for all $j=1,\dots,n$ $P_{acl(E)}(v_j)=P_{acl(E\cup F)}(v_j)$
\erm

\begin{theo}\label{nonforkingextension}
Let $F\subseteq G\subseteq H$, $p\in S_n(F)$ $q\in S_n(G)$ and $\bv=(v_1,\dots,v_n)$, $\bw=(v_1,\dots,v_n)\in H^n$ be such that $p=tp(\bv/F)$ and $q=tp(\bw/G)$. Then $q$ is an extension of $p$ such that $\bw\ind^*_FG$ if and only if the following conditions hold:
    \begin{enumerate}
        \item For every $j=1,\dots,n$, $P_{acl(F)}(v_j)=P_{acl(G)}(w_j)$
        \item For every $j=1,\dots,n$, $\m_{P^\perp_{acl(F)}v_j}=\m_{P^\perp_{acl(G)}w_j}$
    \end{enumerate}
\end{theo}
\begin{proof}
Clear from Theorem \ref{typeoverempty} and Remark \ref{independenceoverA}
\end{proof}

\begin{theo}\label{explicitnonforkingrelation}
$\ind^*$ satisfies:
\ben
	\item Local character.
	\item Finite character.
	\item Transitivity of independence
	\item Symmetry
	\item Existence
	\item Stationarity
\een
\end{theo}
\begin{proof}
By Remark \ref{RemarkItIsEnoughIndependenceForOneVector}, to prove local character, finite character and transitivity it is enough to show them for the case of a 1-tuple. 
\begin{description}
   \item[Local character]
Let $v\in H$ and $G\subseteq \tilde{H}$. Let $w=(P_{acl(G)}(v))_e$. Then there exist a sequence of $(l_k)_{k\in\N}\subseteq \N$, a sequence  $(f_1^k,\dots,f_{l_k}^k)_{k\in\N}$ of finite tuples of bounded Borel funtions of $\R$ and a sequence of finite tuples $(e_1^k,\dots,e_{l_k}^k)_{k\in\N}\subseteq G$ such that if $w_k:=\sum_{j=1}^{l_k}f_j^k(\tilde{Q})e_j^k$ for $k\in\N$, then $w_k\to w$ when $k\to\infty$. Let $E_0=\{e_j^k\ |\ j=1,\dots,l_k\text{ and }k\in\N\}$. Then $v\ind_{E_0}^*E$ and $|E_0|=\aleph_0$.
   \item[Finite character] We show that for $v\in H$, $E,F\subseteq \tilde{H}$, $v\ind^*_EF$ if and only if $v\ind^*_EF_0$ for every finite $F_0\subseteq F$. The left to right direction is clear. For right to left, suppose that $v\nind^*_EF$. Let $w=P_{acl(E\cup F)}(v)-P_{acl(E)}(v)$. Then $w\in$ acl$(E\cup F)\setminus$acl$(E)$. 
   
As in the proof of local character, there exist a sequence of pairs $(l_k,n_k)_{k\in\N}\subseteq \N^2$, a sequence $(g_1^k,\dots,g_{l_k+n_k}^k)_{k\in\N}$ of finite tuples of bounded Borel functions on $\R$, and a sequence of finite tuples $(e_1^k,\dots,e_{l_k}^k,f_1^k,\dots,f_{n_k}^k)_{k\in\N}$ such that $(e_1^k,\dots,e_{l_k}^k)\subseteq E$, $(f_1^k,\dots,f_{n_k}^k)_{k\in\N}\subseteq F$ and if $w_k:=\sum_{j=1}^{l_k}g_j^k(\tilde{Q})e_i^k+\sum_{j=1}^{n_k}g_{l_k+j}^k(\tilde{Q})f_j^k$ for $k\in\N$, then $w_k\to w$ when $k\to\infty$.

If $v\nind^*_EF$, then $w=P_{acl(E\cup F)}(v)-P_{acl(E)}(v)\neq 0$. For $\e=\|w\|>0$ there is $k_\e$ such that if $k\geq k_\e$ then $\|w-w_k\|<\e$. Let $F_0:=\{f_1^1,\dots,f_{k_\e}^{n_{k_\e}}\}$ Then $F_0$ is a finite subset such that $v\nind^*_EF_0$.
   \item[Transitivity of independence] Let $v\in H$ and $E\subseteq F\subseteq G\subseteq H$. If $v\ind_E^*G$ then  $P_{acl(E)}(v)=P_{acl(G)}(v)$. It is clear that $P_{acl(E)}(v)=P_{acl(F)}(v)=P_{acl(G)}(v)$ so $v\ind_E^*F$ and $v\ind_F^*G$.
Conversely, if $v\ind_E^*F$ and $v\ind_F^*G$, we have that $P_{acl(E)}(v)=P_{acl(F)}(v)$ and $P_{acl(F)}(v)=P_{acl(G)}(v)$. Then $P_{acl(E)}(v)=P_{acl(G)}(v)$ and $v\ind_E^*G$.
   \item[Symmetry]  It is clear from Remark \ref{independenceoverA}.
   \item[Invariance] Let $U$ be an automorphism of $(\tilde{H},\G_{\tilde{Q}})$. Let $\bv=(v_1,\dots,v_n)$,$\bw=(w_1,\dots,w_n)\in \tilde{H}^n$ and $G\subseteq \tilde{H}$ be such that $\bv\ind_G^*\bw$. By Remark \ref{independenceoverA}, this means that for every $j$, $k=1,\dots,n$ $\tilde{H}_{P^\perp_{acl(G)}(v_j)}\perp \tilde{H}_{P^\perp_{acl(G)}(w_k)}$. It follows that for every $j$, $k=1,\dots,n$ $\tilde{H}_{P^\perp_{acl(UG)}(Uv_j)}\perp \tilde{H}_{P^\perp_{acl(UG)}(Uw_k)}$ and, again by Remark \ref{independenceoverA}, $Uv\ind_{acl(UG)}^*Uw$.
   \item[Existence] Let $F\subseteq G\subseteq \tilde{H}$ be small sets. We show, by induction on $n$, that for every $p\in S_n(F)$, there exists $q\in S_n(G)$ such that $q$ is an $\ind^*$-independent extension of $p$. 
		\bdc
			\item[Case $n=1$] Let $v\in \tilde{H}$ be such that $p=tp(v/F)$ and let $(H^\prime,Q^\prime)\in\K_{(H,\G_Q)}$ be a structure containing $v$ and $G$. Define 
\[H^{\prime\prime}:=H^\prime\oplus L^2(\R,\m_{(P^\perp_{acl(F)}v)_e}),\]
\[Q^{\prime\prime}:=Q^\prime\oplus M_{f_{(P^\perp_{acl(F)}v)_e}}\]
and
\[v^\prime:=P_{acl(F)}v+P^\perp_{acl(F)}v_d+(1)_{\sim_{\m_{(P^\perp_{acl(F)}v)_e}}}\]
Then $(H^{\prime\prime},Q^{\prime\prime})\in\K_{(H,\G_Q)}$, $v^\prime\in H^{\prime\prime}$ and, by Theorem \ref{nonforkingextension}, the type $gatp(v^\prime/G)$ is a $\ind^*$-independent extension of $tp(v/F)$.
			\item[Induction step] Now, let $\bv=(v_1,\dots,v_n,v_{n+1})\in \tilde{H}^{n+1}$. By induction hypothesis, there are $v_1^\prime,\dots,v_n^\prime\in H$ such that $gatp(v_1^\prime,\dots,v_n^\prime/G)$ is a $\ind^*$-independent extension of $gatp(v_1,\dots,v_n/F)$. Let $U$ be a monster model automorphism fixing $F$ pointwise such that for every $j=1,\dots,n$, $U(v_j)=v_j^\prime$. Let $v_{n+1}^\prime\in \tilde{H}$ be such that $gatp(v_{n+1}^\prime/Gv_1^\prime\cdots v_n^\prime)$ is a $\ind^*$-independent extension of $gatp(U(v_{n+1})/Fv_1^\prime,\cdots v_n^\prime)$. Then, by transitivity, $gatp(v_1^\prime,\dots,v_n^\prime,v_{n+1}^\prime/G)$ is a $\ind^*$-independent extension of $gatp(v_1,\dots,v_n,v_{n+1}/F)$.
		\edc
   \item[Stationarity] Let $F\subseteq G\subseteq \tilde{H}$ be small sets. We show, by induction on $n$, that for every $p\in S_n(F)$, if $q\in S_n(G)$ is a $\ind^*$-independent extension of $p$ to $G$ then $q=p^\prime$, where $p^\prime$ is the $\ind^*$-independent extension of $p$ to $G$ built in the proof of existence.
		\bdc
			\item[Case $n=1$] Let $v\in H$ be such that $p=gatp(v/F)$, and let $q\in S(G)$ and $w\in H$ be such that $w\models q$. Let $v^\prime$ be as in previous item. Then, by Theorem \ref{nonforkingextension} we have that:
 		        \begin{enumerate}
        	   		\item $P_{acl(F)}v=P_{acl(G)}v^\prime=P_{acl(G)}w=$
			    	\item $\m_{P^\perp_{acl(F)}v}=\m_{P^\perp_{acl(G)}w}=\m_{P^\perp_{acl(G)}v^\prime}$
			    \end{enumerate}
This means that $P_{acl(G)}v^\prime=P_{acl(G)}w$, $\m_{P^\perp_{acl(G)}w}=\m_{P^\perp_{acl(G)}v^\prime}$ and, therefore $q=tp(v^\prime/G)=p^\prime$.
			\item[Induction step] Let $\bv=(v_1,\dots,v_n,v_{n+1})$, $\bv^\prime=(v_1^\prime,\dots,v_n,v_{n+1}^\prime)$ and $\bw=(w_1,\dots,w_n)\in\tilde{H}$ be such that $\bv\models p$, $\bv^\prime\models p^\prime$ and $\bw\models q$. By transitivity, we have that $gatp(v_1^\prime,\dots,v_n^\prime/G)$ and $gatp(w_1,\dots,w_n/G)$ are $\ind^*$-independent extensions of $gatp(v_1,\dots,v_n/F)$. By induction hypothesis, $gatp(v_1^\prime,\dots,v_n^\prime/G)=gatp(w_1,\dots,w_n/G)$. Let $U$ be a monster model automorphism fixing $F$ pointwise such that for every $j=1,\dots,n$, $U(v_j)=v_j^\prime$ and let $U^\prime$ a monster model automorphism fixing $G$ pointwise such that for every $j=1,\dots,n$, $U^\prime(v_j^\prime)=w_j^\prime$. Again by transitivity, 
\[gatp(U^{-1}(v_{n+1}^\prime)/Gv_1\cdots v_n)\] 
and 
\[gatp((U^\prime\circ U)^{-1}(w_{n+1})/Gv_1,\cdots v_n)\] 
are $\ind^*$-independent extensions of $gatp(v_{n+1}/Fv_1,\cdots v_n)$. 

By the case $n=1$,
\[gatp(U^{-1}(v_{n+1}^\prime)/Gv_1\cdots v_n)=gatp((U^\prime\circ U)^{-1}(w_{n+1})/Gv_1,\cdots v_n)\]
and therefore 
\[p^\prime=gatp(v_1^\prime,\dots,v_n^\prime v_{n+1}^\prime/G)=gatp(w_1,\dots,w_n,w_{n+1}/G)=q.\]
		\edc
\end{description}
\end{proof}

\bdf
Let $\K$ be an homogeneous MAEC with monster model $\MM$. Let $B\subseteq A\subseteq M$ and let $a\in M$. The type $gatp(a/A)$ is said to \textit{split} over $B$ if there are $b$, $c\in A$ such that
\[gatp(b/B)=gatp(c/B)\]
but
\[gatp(b/Ba)\neq gatp(c/Ba)\]
\edf

\begin{theo}\label{TheoremSplittingImpliesSpectralIndependence}
Let $v\in\tilde{H}$ and let $F\subseteq G\subseteq\tilde{H}$. If $gatp(v/G)$ splits over $F$ then $v\nind^*_F G$.
\end{theo}
\bpf
If $gatp(v/G)$ splits over $F$, then there are two vectors $w_1$ and $w_2\in G$ such that $gatp(w_1/F)=gatp(w_2/F)$ but $gatp(w_1/Fv)\neq gatp(w_2/Fv)$. Then, either $gatp(P_{acl(Fv)}^\perp w_1/\emptyset)\neq gatp(P_{acl(Fv)}^\perp w_2/\emptyset)$ or $P_{acl(Fv)}w_1\neq P_{acl(Fv)}w_2$. Let us consider each case:
\bdc
	\item[Case $gatp(P_{acl(Fv)}^\perp w_1/\emptyset)\neq gatp(P_{acl(Fv)}^\perp w_2/\emptyset)$] 
Since 
\[P_{acl(Fv)}^\perp w_1=P_{acl(F)}^\perp w_1-P_{P^\perp_{acl(F)}v_e}w_1\]
and 
\[P_{acl(Fv)}^\perp w_2=P_{acl(F)}^\perp w_2-P_{P^\perp_{acl(F)}v_e}w_2,\]
this means that
\[gatp(P_{P^\perp_{acl(F)}v_e}w_1/\emptyset)\neq gatp(P_{P^\perp_{acl(F)}v_e}w_2/\emptyset)\]
So, either $P_{P^\perp_{acl(F)}v_e}w_1\neq 0$ or $P_{P^\perp_{acl(F)}v_e}w_2\neq 0$. Let us suppose without loss of generality that $P_{P^\perp_{acl(F)}v_e}w_1\neq 0$. Then $P_{w_1}(P^\perp_{acl(F)}v_e)\neq 0$, which implies that $P_{acl(F)}v\neq P_{acl(Fw_1)}v$. That is, $v\nind^*_Fw_1$ and by transitivity, $v\nind^*_FG$.
	\item[Case $P_{acl(Fv)}w_1\neq P_{acl(Fv)}w_2$] Since 
\[P_{acl(Fv)} w_1=P_{acl(F)} w_1+P_{P^\perp_{acl(F)}v_e}w_1\]
and 
\[P_{acl(Fv)} w_2=P_{acl(F)} w_2+P_{P^\perp_{acl(F)}v_e}w_2,\]
this means that $P_{P^\perp_{acl(F)}v_e}w_1\neq P_{P^\perp_{acl(F)}v_e}w_2$ and, therefore either $P_{P^\perp_{acl(F)}v_e}w_1\neq 0$ or $P_{P^\perp_{acl(F)}v_e}w_2\neq 0$. As in previous item, this implies that $v\nind^*_FG$.
\edc
\epf

\bth\label{Theorem&Spectral&Independence&Implies&Splitting}
Let $v\in \tilde{H}$ and $F\subseteq G\subseteq \tilde{H}$ such that $F=acl(F)$ and $B$ is $|A|$-saturated. If $v\nind_F^*G$, then $v$ splits over $F$.
\eth
\bpf
If $v\nind_F^*G$ then $w:=P_Gv-P_Fv\neq 0$ and $w\perp F$. Since $G$ is $|F|$-saturated, there is $w^\prime\in G$ such that $gatp(w/F)=gatp(w^\prime/F)$ and $w^\prime\perp P_Gv$. Since $\langle v\ |\ w\rangle\neq 0$, $P_vw\neq 0$, while $P_vw^\prime=0$. 
\epf

\begin{defin}\label{AepsilonindependentB}
Let $\e>0$, $v\in\tilde{H}$ and let $F$, $G\subseteq\tilde{H}$. We say that $v$ is $\e$-\textit{spectrally independent} from $G$ over $F$ if $P_{acl(F\cup G)}v-P_{acl(F)}v\leq \e$ and denote it $v\ind^\e_F G$.
\end{defin}

\begin{theo}\label{PropertiesOfEpsilonIndependence}
The relation $\ind^\e$  satisfies the following properties:
\bdc
	\item[Local character] Let $v\in H$, $G\subseteq \tilde{H}$ and $\e>0$. Then there is a finite $G_0\subseteq G$ such that $v\ind_{G_0}^\e G$.
	\item[Transitivity of independence] Let $v\in H$ and $D\subseteq E\subseteq F\subseteq G\subseteq H$. If $v\ind_D^\e G$ then $v\ind_E^\e F$
\edc
\end{theo}
\bpf
\bdc
	\item[Local character] Let $v\in H$, $G\subseteq \tilde{H}$ and $\e>0$. Let $w$,  $(l_k)_{k\in\N}\subseteq \N$, $(e_1^k,\dots,e_{l_k}^k)_{k\in\N}\subseteq G$, $(f_1^k,\dots,f_{l_k}^k)_{k\in\N}$ and $w_k$ for $k\in\N$ be as in the proof of local character of $\ind^*$ in Theorem \ref{explicitnonforkingrelation}. Since $w_k\to w$ when $k\to\infty$, there is a $k_1\in\Z$ such that $\|w_k-w\|<\e$ for all $k\geq k_1$. Let $G_o:=\{e_j^k\ |\ j=1,\dots,l_k\text{ and }k\leq k_1\}$. Then, $v\ind_{G_0}^*G$.
	\item[Transitivity of independence] Let $v\in H$ and $D\subseteq E\subseteq F\subseteq G\subseteq H$ and $\e>0$. If $v\ind_D^\e G$ then $\e\geq P_{acl(D\cup G)}v-P_{acl(D)}v=P_{acl(G)}v-P_{acl(D)}v\geq P_{acl(f)}v-P_{acl(E)}v$. Therefore $v\ind_E^\e F$.
%
\edc
\epf


\bdf
Let $\bv=(v_1,\dots,v_n)\in H^n$ and $G\subseteq H$. A \textit{canonical base} for the type $gatp(\bv/G)$ is a set $F\subseteq H_G$ such that $\bv\ind^*_FG$ and $|F|$ is minimal.
\edf

\begin{theo}\label{canonicalbase}
Let $\bv=(v_1,\dots,v_n)\in H^n$ and $G\subseteq H$. Then $Cb(gatp(\bv/G)):=\{(P_Gv_1,\dots,P_Gv_n)\}$ is a canonical base for the type $gatp(\bv/G)$
\end{theo}
\begin{proof}
First of all, we consider the case of a 1-tuple. By Theorem \ref{nonforkingextension} $gatp(v/G)$ does not fork over $Cb(gatp(v/G))$. Let $(v_k)_{k<\omega}$ a Morley sequence for $gatp(v/G)$. We have to show that $P_Gv\in dcl((v_k)_{k<\omega})$. By Theorem \ref{nonforkingextension}, for every $k<\omega$ there is a vector $w_k$ such that $v_k=P_Gv+w_k$ and $w_k\perp acl(\{P_Gv\}\cup\{w_j\ |\ j< k\})$. This means that for every $k<\omega$, $w_k\in H_e$ and for all $j$, $k<\omega$, $H_{w_j}\perp H_{w_k}$. For $k<\omega$, let $v^\prime_k:=\frac{v_1+\cdots+v_k}{n}=P_Gv+\frac{w_1+\cdots+w_k}{n}$. Then for every $k<\omega$, $v^\prime_k\in dcl((v_k)_{k<\omega})$. Since $v^\prime_k\to P_ev$ when $k\to\infty$, we have that $P_Gv\in dcl((v_k)_{k<\omega})$.

For the case of a general $n$-tuple, by Remark \ref{RemarkItIsEnoughIndependenceForOneVector}, it is enough to repeat previous argument in every component of $\bv$.
\end{proof}

\section{orthogonality and domination}\label{orthogonalitydomination}
In this section, we characterize domination, orthogonality of types in terms of absolute continuity and mutual singularity between spectral measures. This is done in Corollary \ref{orthogonalityoverA} and Corollary \ref{dominationoverA}.

\begin{theo}\label{orthogonalityoverempty}
Let $p,q\in S_1(\emptyset)$, let $v\models p$ and $w\models q$.
Then, $p\perp^a q$ if and only if $\m_{v_e}\perp \m_{w_e}$.
\end{theo}
\begin{proof}
$p\perp^a q$ if and only if $\tilde{H}_{v_{e}^\prime}\perp \tilde{H}_{w_{e}^\prime}$ for all $v_{e}^\prime\models p$ and $w_{e}^\prime\models q$. By Lesbesgue decomposition theorem $\m_{w_e}=\m_{v_e}^\parallel+\m_{v_e}^\perp$ where, $\m_{v_e}^\parallel <<\m_{v_e}$ and $\m_{v_e}^\perp\perp\m_{v_e}$. $\m_{v_e}^\parallel\neq 0$ if and only if there is a choice of $v^\prime\models p$ and $w^\prime\models q$ such that $\tilde{H}_{v_e^\prime}\cap \tilde{H}_{w_e^\prime}\neq \{0\}$ and therefore $\tilde{H}_{v_e^\prime}\not\perp \tilde{H}_{w_e^\prime}$. 
\end{proof}

\begin{coro}\label{orthogonalityoverA}
Let $G\subseteq \tilde{H}$ be small. Let $p,q\in
S_1(G)$, let $v\models p$ and $w\models q$. Then, $p\perp_G^a q$ if and only if $\m_{P^\perp_Gv_e}\perp \m_{P^\perp_Gw_e}$
\end{coro}
\begin{proof}
Clear from Theorem \ref{orthogonalityoverempty}.
\end{proof}

\begin{coro}
Let $G\subseteq H$ be small. Let $p,q\in
S_1(G)$. Then, $p\perp^a q$ if and only if $p\perp q$.
\end{coro}
\bpf
Clear from Corollary \ref{orthogonalityoverA}.
\epf

\begin{theo}\label{dominationoverempty}
Let $p,q\in S_1(\emptyset)$, let $v\models p$ and $w\models q$.
Then, $p\vartriangleright_\emptyset q$ if and only if $\m_{v_e}>>\m_{w_e}$.
\end{theo}
\begin{proof}
Suppose $p\vartriangleright_\emptyset q$. Suppose that $v$ and $w$ are such that if $v\ind^*_\emptyset G$ then $w\ind_\emptyset^*G$ for every $G\subseteq\tilde{H}$. Then for every $G$ if $\tilde{H}_{v_e}\perp \tilde{H}_G$ then $\tilde{H}_{w_e}\perp \tilde{H}_G$. This means $\tilde{H}_{w_e} \subseteq \tilde{H}_{v_e}$ and $\tilde{H}_{w_e}$ is unitarily equivalent to some Hilbert subspace of $\tilde{H}_{w_e}$ and by Theorem \ref{HvinHw} $\m_{w_e} <<\m_{v_e}$.
\end{proof}

\begin{coro}\label{dominationoverA}
Let $E$, $F$, and $G$ be small subsets of $\tilde{H}$ and $p\in S_1(F)$ and $q\in S_1(G)$ two stationary types. Then $p\vartriangleright_E q$ if and only if there exist $v$ $w\in\tilde{\H}$ such that $gatp(v/E)$ is a non-forking extension of $p$, $gatp(w/E)$ is a non-forking extension of $q$ and $\m_{P^\perp_{acl(F)}v}>>\m_{P^\perp_{acl(F)}w}$.
\end{coro}
\begin{proof}
Clear from previous theorem.
\end{proof}

\end{document}